\documentclass[reqno]{amsart}
\usepackage[a4paper,hmarginratio=1:1]{geometry}
\usepackage{amssymb,amsfonts,amsmath}
\usepackage{comment} 
\usepackage[all,arc]{xy}
\usepackage{enumerate}
\usepackage{mathrsfs,mathtools}
\usepackage{todonotes,booktabs}
\usepackage{stmaryrd}
\usepackage{marvosym}
\usepackage{graphicx}
\usepackage{pdflscape} 
\usepackage{scalerel}
\usepackage{floatrow}
\usepackage{caption}
\usepackage{varwidth}
\usepackage{wrapfig}
\usepackage{circledsteps}

\usepackage[parfill]{parskip}

\usepackage{bbm}
\usepackage{microtype}
\usepackage{dsfont}
\usepackage{enumitem}

\usepackage{tikz}
\usepackage{tikz-cd}
\usetikzlibrary{trees}
\usetikzlibrary[shapes]
\usetikzlibrary[arrows]
\usetikzlibrary{patterns}
\usetikzlibrary{patterns.meta}
\usetikzlibrary{fadings}
\usetikzlibrary{backgrounds}
\usetikzlibrary{decorations.pathreplacing}
\usetikzlibrary{decorations.pathmorphing}
\usetikzlibrary{positioning}

\usepackage{xcolor} 
\usepackage{graphicx}
\SelectTips{cm}{10}

\usepackage{hyperref}

\definecolor{dark-red}{rgb}{0.5,0.15,0.15}
\definecolor{dark-blue}{rgb}{0.15,0.15,0.6}
\definecolor{dark-green}{rgb}{0.05,0.8,0.05}
\hypersetup{
    colorlinks, linkcolor=dark-red,
    citecolor=dark-blue, urlcolor=dark-green
}

\usepackage[nameinlink,capitalise,noabbrev]{cleveref}
 
\newtheorem{thm}{Theorem}[section]
\newtheorem{theorem}[thm]{Theorem}
\newtheorem{corollary}[thm]{Corollary}
\newtheorem{proposition}[thm]{Proposition}
\newtheorem{lemma}[thm]{Lemma}

\newtheorem*{theorem*}{Theorem}
\newtheorem*{proposition*}{Proposition}
\newtheorem*{remark*}{Remark}
\newtheorem*{conjecture*}{Conjecture}

\theoremstyle{definition}
\newtheorem{definition}[thm]{Definition}
\newtheorem{example}[thm]{Example}

\newtheorem{construction}[thm]{Construction}

\theoremstyle{remark}
\newtheorem{remark}[thm]{Remark}

\usepackage{stackengine}
\usepackage{scalerel}

\makeatletter
\let\c@equation\c@thm
\makeatother
\numberwithin{equation}{section}

\newcommand{\cpctRecollement}[3]{
\xymatrix@C=2em{{#1} \ar[r]|-{#2} & {#3}}
}

\newcommand{\recollement}[5]{
\xymatrix@C=4em{{#1} \ar@{<-}[r]|-{#2} & #3 \ar@{<-}[r]|-{#4} \ar@{<-}@<1.5ex>[l]^-{{#2}_!} \ar@{<-}@<-1.5ex>[l]_-{{#2}^*} & #5, \ar@{<-}@<1.5ex>[l]^-{{#4}!} \ar@{<-}@<-1.5ex>[l]_-{{#4}^*}
}}
\let\lim\relax

\DeclareMathOperator{\lim}{lim}

\newcommand{\F}{\mathbb{F}}
\newcommand{\N}{\mathbb{N}}
\newcommand{\Q}{\mathbb{Q}}
\newcommand{\R}{\mathbb{R}}
\newcommand{\Z}{\mathbb{Z}}

\DeclareMathOperator{\cI}{\mathcal{I}}
\DeclareMathOperator{\cJ}{\mathcal{J}}

\newcommand{\fj}{\mathfrak{j}}
\newcommand{\fm}{\mathfrak{m}}
\newcommand{\fn}{\mathfrak{n}}
\newcommand{\fp}{\mathfrak{p}}
\newcommand{\fq}{\mathfrak{q}}

\DeclareMathOperator{\Spc}{Spc}
\DeclareMathOperator{\Hom}{Hom}

\DeclareMathOperator{\Tor}{Tor}
\DeclareMathOperator{\Spec}{Spec}

\DeclareMathOperator{\supp}{supp}

\DeclareMathOperator{\CRings}{CRings}

\newcommand{\sfT}{\mathsf{T}}

\newcommand{\sfD}{\mathsf{D}}

\newcommand{\Ab}{\mathrm{Ab}}

\newsavebox\prismsym
\savebox\prismsym{\begin{tikzpicture}\draw[thick] (0,-0.1) -- (0.2,1);
\draw[thick, dotted] (-0.3,0.3) -- (0.7,0.3);
\draw[ultra thick] (0,-0.1) -- (-0.3,0.3);
\draw[ultra thick] (0,-0.1) -- (0.7,0.3);
\draw[ultra thick] (0.2,1) -- (0.7,0.3);
\draw[ultra thick] (0.2,1) -- (-0.3,0.3);\end{tikzpicture}}

\definecolor{darkBlue}{rgb}{0.0, 0.18, 0.65}
\usepackage{tabularx}

\newcommand{\idem}{\mathrm{idem}}
\newcommand{\Next}{\mathrm{next}}
\newcommand{\home}{\Hom_{\mathrm{epi}}}
\newcommand{\flate}{\Hom_{\mathrm{flat}}}
\newcommand{\Sm}{\mathsf{Sm}}
\newcommand{\Th}{\mathsf{Th}}
\newcommand{\sfI}{\mathsf{I}}
\newcommand{\sfJ}{\mathsf{J}}
\newcommand{\fraki}{\mathfrak{i}}
\newcommand{\inter}{\mathrm{Inter}}
\newcommand{\interc}{\mathrm{Inter}_{\mathrm{Ch}}}
\newcommand{\ontop}[2]{\stackrel{\textrm{\tiny #1}}{#2}} 
\newcommand{\dert}{\otimes^\mathbf{L}}
\tikzset{mymatr/.style={every outer matrix/.append style={draw=black, inner xsep=4pt , inner ysep=6pt, rounded corners, very thick}}}

\newfloatcommand{capbtabbox}{table}[][\FBwidth] 
\newfloatcommand{captikzbox}{tikzcd}[][\FBwidth] 
\newcolumntype{P}[1]{>{\centering\arraybackslash}p{#1}}
\newsavebox{\tikzcdbox}

\usepackage{tkz-graph}

\title{Classifying smashing ideals in derived categories of valuation domains}

 \author{Scott Balchin}
 \address{Mathematical Sciences Research Centre, Queen's University Belfast, UK}
  \email{s.balchin@qub.ac.uk }

\author{Florian Tecklenburg}
 \address{Mathematisches Institut, Universit\"{a}t Bonn, Germany}
  \email{s6flteck@uni-bonn.de }

\date{\today}
\begin{document}
\begin{abstract}
    Building on results of Bazzoni--{\v{S}}{t}\hspace{-0.125em}{'}\hspace{-0.125em}ov\'{\i}\v{c}ek, we give a complete classification of the frame of smashing ideals for the derived category of a finite dimensional valuation domain. In particular, we give an explicit construction of an infinite family of commutative rings such that the telescope conjecture fails and which generalise an example of Keller.  As a consequence, we deduce that the Krull dimension of the Balmer spectrum and the Krull dimension of the smashing spectrum can differ arbitrarily for rigidly-compactly generated tensor-triangulated categories.
\end{abstract}

\maketitle

\section{Introduction}

Let $\sfT$ be a rigidly-compactly generated tensor-triangulated category of interest, say the derived category of a commutative ring $\sfD(R)$, the stable module category of $kG$ modules $\mathsf{StMod}(kG)$, or the stable homotopy category $\mathsf{Sp}$. Then it is a natural question to ask how much information is lost under passage to the collection of compact objects $\sfT^\omega$. One way to measure the amount of information which is lost is to identify which smashing ideals are generated by compact objects. If no information is lost, that is, if all smashing ideals are compactly generated, then one says that the \emph{telescope conjecture} holds for $\sfT$.

There are many cases where we know that the telescope conjecture holds, for example we know it holds for the derived category of a commutative Noetherian ring, for the derived category of an absolutely flat ring, for the stable module category, and the category of rational $G$-equivariant spectra where $G$ is either compact Lie or profinite \cite{neemanchromtower,stevenson_absolutelyflatrings,hrbek2024telescopeconjecturehomologicalresidue, bik11,balchin2023prismatic,balchin2024profinite}. However, the collection of examples where the telescope conjecture is known to fail is surprisingly sparse. Keller produced the first example of a commutative ring $A$ such that $\sfD(A)$ fails the telescope conjecture \cite{keller94}, and it was recently proved by Burkland--Hahn--Levy--Schlank that the telescope conjecture fails for $\mathsf{Sp}$ \cite{burklund2023ktheoretic}. Worse still, there are even fewer cases where we have an explicit description of the non-compactly generated smashing ideals when the telescope conjecture fails.  We have a description of these ideals in the case of Keller's ring, and in the non rigidly-compactly generated case, Aoki has recently provided an example of a tensor-triangulated category $\sfT$ such that the collection of non-compactly generated smashing ideals is highly exotic \cite{aoki2024smashing}. The goal of this paper is to move towards remedying this lack of explicit examples in the simple case of a derived category of a commutative ring.

More concretely, one can associate to $\sfT^\omega$ the collection of thick ideals $\Th(\sfT^\omega)$. This forms a frame which is moreover spatial \cite{kockpitsch}, and thus by Stone duality there is an associated topological space whose frame of open sets retrieves $\Th(\sfT^\omega)$. We denote this space by $\Spc(\sfT^\omega)^\vee$. This space is the Hochster dual of the Balmer spectrum of $\sfT^\omega$ \cite{balmer_spectrum, balmer_3spectra}, and has been computed in a multitude of examples. In particular, the points of this space correspond to the prime thick ideals. Moving back to the ambient large category $\sfT$, one can instead consider the collection of smashing ideals $\Sm(\sfT)$ -- which once again forms a frame \cite{bks20}. However, it is still an open question whether this frame is spatial in general \cite{balchin2023big}. In the case where this frame is indeed spatial (for example when it is finite, as it will be for us), we once again can apply Stone duality to obtain a topological space $\Spc^\mathrm{s}(\sfT)$, the \emph{smashing spectrum} of $\sfT$ whose frame of open sets is exactly $\Sm(\sfT)$. There is a frame homomorphism $\mathrm{infl} \colon \Th(\sfT) \to \Sm(\sfT)$ which inflates a thick ideal to a smashing ideal.  The statement of the telescope conjecture can be formulated by asking that this homomorphism is an isomorphism. When $\Sm(\sfT)$ is spatial, this is equivalent to asking that the induced map $\Spec(\mathrm{infl})$ is a homeomorphism of spaces. We note that $\Spec(\mathrm{infl})$ can be identified with the map that takes a prime smashing ideal $\mathsf{P}$ to the thick prime ideal $\mathsf{P} \cap \sfT^\omega$ \cite[Theorem 5.1.2]{balchin2023big}.  

In this paper we build on the work of Bazzoni--{\v{S}}{t}\hspace{-0.125em}{'}\hspace{-0.125em}ov\'{\i}\v{c}ek \cite{gldm1} to generalise the aforementioned example of Keller to an infinite family of commutative rings $A^{\star n}$, for which $\sfD(A^{\star n})$ does not satisfy the telescope conjecture, and yet we can still give a complete characterisation of the frame $\Sm(\sfD(A^{\star n}))$. In our setting, the latter is always finite and thus spatial. These examples allow us to observe new ways in which the telescope conjecture can fail. A summary of the  behaviour of the rings $A^{\star n}$ is the following, the proofs of which appear in \cref{sect:generalized}:

\begin{theorem*}
    Let $n \geq 1$. Then there is a valuation domain of Krull dimension $n$, $A^{\star n}$, such that:
    \begin{enumerate}
        \item $|\Th(\sfD(A^{\star n})^\omega)| = n + 2$.
        \item $\Spc(\sfD(A^{\star n})^\omega)^\vee \cong [n+1]$ equipped with the Alexandroff topology.
        \item $|\Sm(\sfD(A^{\star n}))| = \mathrm{Fib}_{2n+3}$ where $\mathrm{Fib}_i$ is the $i^\mathrm{th}$ Fibonacci number.
        \item 
        $
        \Spc^\mathrm{s}(\sfD(A^{\star n})) =
\begin{tikzcd}[style={every outer matrix/.append style={draw=black, inner xsep=6pt , inner ysep=9pt, rounded corners, very thick}},row sep=0.4cm,column sep=0.3cm]
c_1 && c_2 && c_3 &  \cdots & c_n && c_{n+1}\\
& o_1 \arrow[ru, rightsquigarrow] \arrow[lu, rightsquigarrow] && \arrow[ru, rightsquigarrow] \arrow[lu, rightsquigarrow]  o_2 &&{\quad \cdots \quad } \arrow[ru, rightsquigarrow] \arrow[lu, rightsquigarrow]&& \arrow[ru, rightsquigarrow] \arrow[lu, rightsquigarrow]  o_n \end{tikzcd}
        $
        
        where the $c_i$ are closed points, the $o_j$ are open points, and an arrow denotes specialization closure.
    \end{enumerate}
\end{theorem*}

We highlight that by these results, we are able to generate examples where there are arbitrarily more smashing ideals than thick ideals, and more interestingly still, examples where the Balmer spectrum has arbitrarily high Krull dimension while the smashing spectrum is always one-dimensional.

Further still, we can extend the above theorem to fully classify those frames $\Sm(\sfD(R))$ where $R$ is a finite dimensional valuation domain. The rings $A^{\star n}$ play a special role in this classification, in particular we will prove the following: 

\begin{theorem*}[\cref{thm:mixed_subframe}]
Let $R$ be an $n$-dimensional valuation domain. Then $\Sm(\sfD(R))$ is an explicit subframe of $\Sm(\sfD(A^{\star n}))$.
\end{theorem*}

In addition to this theorem, we moreover provide a description of the corresponding spectral space via Stone duality in \cref{thm:smashing_spectrum_mixed}.

In \cref{sect:VD} we introduce the relevant theory regarding valuation domains. In particular, we recall how to construct a valuation domain from a totally ordered group. This then leads to \cref{prop:constructionOnRingLevel}, where we introduce our main tool that allows an explicit description of the rings $A^{\star n}$. Continuing  in \cref{sect:stovicek}, we  recall some key results of \cite{gldm1}, which give tight control over the frame of smashing ideals in the derived category of a valuation domain. In \cref{sect:generalized} we come to our promised construction of $A^{\star n}$ and prove the aforementioned results. Finally, in \cref{sect:mixing} we move to the case of an arbitrary finite dimensional valuation domain.

We refer the reader to \cite{balmer_spectrum, balchin2023big, kockpitsch} for the relevant notions from tensor-triangular geometry that we will use throughout, and to \cite{PPframes} for the theory of frames and Stone duality. We will freely use the fact that for a commutative ring $R$, the Zariski spectrum of $R$ is homeomorphic to the Balmer spectrum of $\sfD(R)^\omega$ \cite{neemanchromtower, thomasonclassification}.
 
\subsection*{Acknowledgements}

The first author would like to thank the Max Planck Institute for Mathematics for its hospitality where some of this work was undertaken. The authors also thank Jordan Williamson for helpful comments on a preliminary draft of this paper. Parts of this work were established in the Master's thesis of the second author at the University of Bonn.

\section{Valuation domains from totally ordered groups}\label{sect:VD}

\subsection{The construction and link between the structures}

Before constructing a valuation domain with prescribed value group, we recall the results that we will need from \cite{fs01}.

\begin{definition}\label{def:valuations}
An abelian group $(G,+)$ is {\it totally ordered} if the underlying set $G$ is totally ordered with order $\le$, and the group operation is monotone with respect to this order. Denote by $G^+\coloneqq\{g\in G\mid g\ge0\}$ the {\it positivity domain} of $G$.  

A \textit{morphism of totally ordered groups} is an order-preserving group homomorphism. We will denote by $\Ab_\le$ the category of totally ordered abelian groups. 
\end{definition}

\begin{definition}
A commutative domain $R$ is a {\it valuation domain} if its poset of ideals is totally ordered. 

For a totally ordered group $(G,+,\le)$ and a field $Q$, a map $\nu\colon Q\to G\cup\{\infty\}$ is said to be a {\it valuation} if the following holds:
\begin{enumerate}
\item $\nu(x)=\infty$ for $x\in Q$ if and only if $x=0$.
\item $\nu(xy)=\nu(x)+\nu(y)$ for all $x,y\in Q$.
\item $\nu(x+y)\ge\mathrm{min}(\nu(x),\nu(y))$ for all $x,y\in Q$.
\end{enumerate}

Denote by $R_\nu\coloneqq\{x\in Q\mid \nu(x)\ge0\}$ the {\it valuation ring of} $\nu$. It defines a valuation domain with unique maximal ideal $\fm_\nu\coloneqq\{x\in Q\mid \nu(x)>0\}$. We call $G$ the {\it value group} of $R_\nu$. 

Two valuations $\nu,\mu$ over the same field with value groups $G,H$ are {\it equivalent} if there exists an isomorphism of totally ordered groups $f\colon G\to H$, such that $\mu=f\circ \nu$ (where $\infty$ is fixed by $f$).
\end{definition}

The following classical result gives a full characterization of valuation domains.

\begin{proposition}[{\cite[\S II, Theorem 3.1]{fs01}}]\label{prop:existenceofvaluation}
    Every valuation domain $R$ is canonically isomorphic to the valuation ring $R_\nu$, for some valuation $\nu$ of its quotient field $Q=R_{(0)}$. 
\end{proposition}

Conversely, it is possible to construct valuation domains with a prescribed value group. We provide the proof of this result as the construction will be central to our results.

\begin{proposition}[{\cite[\S1]{kru32}}]\label{prop:prescribed}
Let $k$ be a field and $(G,+,\le)$ be a totally ordered group. Then there exists a valuation domain $R_\nu$ with valuation $\nu$, residue field $R_\nu/\fm_\nu\cong k$, and a value group which is order-isomorphic to $G$.
\end{proposition}
\begin{proof}
Consider $S=k[t^G]$, the ring consisting of all formal polynomials $f(t)=\sum_{i=0}^na_it^{g_i}$, where $a_i\in k$, $g_i\in G$, and the usual ring structure for polynomials. In particular, $t^0$ is identified with $1\in k$ and for any $g,h\in G$ we have $t^g\cdot t^h=t^{g+h}$. Then $S$ is an integral domain, and we claim its field of fractions $Q=S_{(0)}$ can be endowed with a valuation $\nu\colon Q\to G\cup\{\infty\}$ as follows: For $f\in S$ as expressed above set
\[
\nu(f/1)\coloneqq\min_{a_i\neq0}g_i,
\]
and for a fraction $f/g\in Q$ set $\nu(f/g)\coloneqq\nu(f/1)-\nu(g/1)$. It is immediate that this map attains infinity if and only if the fraction is zero. By definition of $\nu$ we obtain for $f,g\in S$ the equality $\nu(f/1\cdot g/1)=\nu(f/1)+\nu(g/1)$. Combining this equality with the definition of $\nu$ on  $Q$, \Cref{def:valuations}(2) is immediate. Moreover, since adding polynomials simply adds their coefficients, one checks that for every $f,g\in S$ that $\nu(f/1+g/1)\ge\min(\nu(f/1),\nu(g/1))$. Using this inequality, we deduce for $f/g,u/v\in Q$:
\begin{align*}
\nu\left(\frac{f}{g}+\frac{u}{v}\right)&=\nu\left(\frac{fv+ug}{gv}\right)=\nu\left(\frac{fv}{1}+\frac{ug}{1}\right)-\nu\left(\frac{gv}{1}\right)\ge\min\left(\nu\left(\frac{fv}{1}\right),\nu\left(\frac{ug}{1}\right)\right)-\nu\left(\frac{gv}{1}\right)\\
&=\min\left(\nu\left(\frac{f}{1}\right)+\nu\left(\frac{v}{1}\right),\nu\left(\frac{u}{1}\right)+\nu\left(\frac{g}{1}\right)\right)-\nu\left(\frac{g}{1}\right)-\nu\left(\frac{v}{1}\right)\\
&=\min\left(\nu\left(\frac{f}{1}\right)-\nu\left(\frac{g}{1}\right),\nu\left(\frac{u}{1}\right)-\nu\left(\frac{v}{1}\right)\right)=\min\left(\nu\left(\frac{f}{g}\right),\nu\left(\frac{u}{v}\right)\right).
\end{align*}
Thus $\nu\colon Q\to G\cup\{\infty\}$ defines a valuation. Hence, $R_\nu=\{x\in Q\mid \nu(x)\ge0\}$ is a valuation domain with maximal ideal $\fm_\nu=\{x\in Q\mid \nu(x)>0\}$. The residue field $R_\nu/\fm_\nu$ is then isomorphic to $k$, and the value group of $R_\nu$ is $G$ as desired.
\end{proof}
\begin{definition}
    We will denote the constructed valuation domain of \Cref{prop:prescribed} by $R_k(G)$, its associated field of fractions by $Q_k(G)$, and the corresponding valuation by $\nu_{Q(G)}$. If there is no danger of confusion about the residue field, we shall omit it, i.e., write $R(G)$ and $Q(G)$.
\end{definition}

The next result tells us that the construction of $R(G)$ is functorial in the appropriate sense.

\begin{lemma}\label{lemma:functorialityVDconstruction}
Let $f\colon G\to H$ be a morphism of totally ordered abelian groups, and fix some field $k$. Then the construction of \Cref{prop:prescribed} is functorial and natural in the following sense: The map 
\begin{align*}
Q(f)\colon Q(G)&\to Q(H)\\[2pt]
\frac{\sum_{i=0}^na_it^{g_i}}{\sum_{i=0}^mb_it^{h_i}}&\mapsto\frac{\sum_{i=0}^na_it^{f(g_i)}}{\sum_{i=0}^mb_it^{f(h_i)}}
\end{align*}
is a morphism of fields and restricts to a well-defined morphism of rings between $R(G)$ and $R(H)$. Hence, there is a functor $R\colon\Ab_\le\to\CRings$. Moreover, let $\nu$ (resp., $\mu$) be the valuation constructed through \Cref{prop:prescribed} from $G$ (resp., $H$). Then the following diagram is commutative:
\[
\begin{tikzcd}
Q(G) \arrow[r, "\nu"] \arrow[d, "Q(f)"'] & G\cup\{\infty\} \arrow[d, "f\cup\{\infty\}"] \\
Q(H) \arrow[r, "\mu"']                   & H\cup\{\infty\}     \rlap{ .}              
\end{tikzcd}
\]
\end{lemma}

\begin{proof}
    Verifying that $Q(f)$ is a homomorphism of fields follows from the fact that $f$ is a homomorphism of groups. To show that its restriction to $R(G)$ maps to $R(H)$ requires that $f$ is a morphism of totally ordered abelian groups, and is deduced immediately from the definition of the valuations on $Q(G)$ and $Q(H)$.

    Finally, let us verify that the depicted diagram in the statement is commutative. Without loss of generality, assume $u/v\in Q(G)$ satisfies $\nu(u/v)=\nu(u)-\nu(v)=g_1-h_1$. As $f$ is order preserving, we know by definition of $\mu$ that $\mu(Q(f)(u/v))=f(g_1)-f(h_1)$. Hence, the diagram commutes for those elements $u/v\in Q(G)$ which have finite valuation in $G$, and the remaining case $u/v=0$ is immediate.
\end{proof}

There is an intimate link between the ideals of a valuation domain and its value group through the valuation. To make this structural connection precise, we need an analogue of ideals in terms of the value group. We remind the reader that $G^+\coloneqq\{g\in G\mid g\ge0\}$.

\begin{definition}
Let $(G, + , \leq)$ be a totally ordered group. A {\it filter} in $G^+$ is a non-empty subset $F\subset G^+$, that is upwards closed, i.e., $a\in F$ and $a\le b\in G^+$ implies $b\in F$. 

Moreover, a filter $F$ is 
\begin{itemize}
\item {\it principle} if $F=\{b\in G^+\mid b\ge a\}$ for some $a\in F$, i.e., $F$ has a unique minimal element;
\item {\it idempotent} if $F+F=F$, i.e., the group operation of $G$ restricts to $F$, is well-defined, and surjective.
\end{itemize}
Finally, a proper filter $F$ in $G^+$ is {\it prime} if for all $a,b\in G^+\setminus F$ we have $a+b\in G^+\setminus F$.
\end{definition}

\begin{proposition}\label{prop:idealsAndFilters}
Let $R=R_\nu$ be a valuation domain with valuation $\nu\colon R_{(0)}\to G\cup\{\infty\}$. For an ideal $I\subseteq R$ and a filter $F\subset G^+$, consider the following assignments:
\[
I\mapsto \nu(I)=\{\nu (r/1)\mid r\in I\},\enspace F\mapsto I(F)\coloneqq\{r\in R\mid \nu(r/1)\in F\}.
\]
Then they define mutually inverse order-preserving maps between the non-zero ideals in $R$ and the set of filters in $G^+$ with respect to inclusion. Moreover, this bijection restricts to non-zero prime (resp., principle, idempotent) ideals, which correspond to prime (resp., principle, idempotent) filters in $G^+$.
\end{proposition}
\begin{proof}
By \cite[\S II, Proposition 3.4]{fs01} the assignments are well-defined, order-preserving, and mutually inverse.
The correspondence of prime filters and prime ideals follows from the definitions and \Cref{def:valuations}(2). Moreover, an ideal $(r/1)\subseteq R$ is assigned to the filter $\{b\in G^+\mid b\ge \nu(r/1)\}$.

Let us discuss how the bijection restricts to the idempotent objects. For a filter $F\subseteq G^+$ let $I=I(F)$, and assume that $I$ is idempotent. Since $\nu(I)=F$, for every $g\in F$ there is some $r/1\in I$ such that $\nu(r/1)=g$. As $I$ is idempotent, we can write $r=\sum_{i=0}^na_ir_i$ for some $a_i,r_i\in I$. By definition of $\nu$ we have
\[
\nu(r/1)\ge\min_{0\le i\le n}\nu(a_ir_i/1)=\min_{0\le i\le n}\nu(a_i/1)+\nu(r_i/1)=\min_{0\le i\le n}\nu(a_i/1)+\min_{0\le i\le n}\nu(r_i/1),
\]
where the last equality uses that all valuations are positive. Since the right-hand side is an element of $\nu(I)+\nu(I)=F+F$, which is upwards closed, we deduce that $g\in F+F$. The converse inclusion $F+F\subseteq F$ is immediate.

Finally, let $F$ be an idempotent filter and set $I=I(F)$. For any $r\in I$ we have $\nu(r/1)\in F$, thus there are $f=\nu(s/1)$, $g=\nu(t/1)$ with $s,t\in I$ satisfying $\nu(r/1)=f+g=\nu(st/1)$. By classical results of valuation rings (e.g., \cite[\S II.3 Exercise 3]{fs01}), this holds if and only if $(r)=(st)\subset R$. Since $st\in I^2$, we have $(r)=(st)\subset I^2$, which proves that $r\in I^2$. Hence, $I\subseteq I^2$, and the converse inclusion is immediate.
\end{proof}

Combining \Cref{prop:prescribed} and \Cref{prop:idealsAndFilters}, we notice that the ideal structure of valuation domains is determined by their  associated value group $G$. Since we can construct valuation domains with prescribed value group, one can reduce all  phenomena which occur on the level of ideals of valuation domains to totally ordered groups. This will be a key observation that we use in \cref{sect:generalized}.

\indent
\subsection{Operations on groups}
Assembling new (totally ordered) groups from given ones is a common theme to construct new examples, and with view towards \Cref{prop:prescribed}, one may ask how such operations act on the resulting valuation domain.

\begin{example}
    Given totally ordered groups $(G,+,\le_G)$, $(H,+,\le_H)$, one can consider their direct sum $G\oplus H$ endowed with the {\it lexicographical order}. This ordering is defined as
    \[
    G\oplus H\ni(g,h)>0 \; \Longleftrightarrow\; \begin{cases}h>_H0, &\mbox{ if } g=0,\\g>_G0,& \mbox{ else}.\end{cases}
    \]
    In other words, an element is strictly positive if and only if its first non-zero entry is strictly positive. More generally, we have $(g,h)<(g^\prime, h^\prime)$ if and only if $(g^\prime-g, h^\prime-h)>0$. Since $G$ and $H$ are totally ordered, one can verify that this  order on the direct sum is again total.
\end{example}

\begin{remark}
We warn that this operation is not commutative. Consider the two totally ordered groups $G_1=\R\oplus \Z$ and $G_2=\Z\oplus \R$ equipped with the lexicographical ordering. We claim that these are not order-isomorphic groups. 

To this end, we will prove that $G_1^+\setminus\{0\}$ is a non-idempotent filter, whilst $G_2^+\setminus\{0\}$ is idempotent. Evidently they are filters, and every element $(n,x)\in G_2^+\setminus\{0\}$ can be written as the sum of two such. On the other hand, the element $(0,1)\in G_1^+\setminus\{0\}$ cannot be written as the sum of two elements, as otherwise $(m,y)+(k,z)=(0,1)$ implies that either $m,k$ are both $0$, or $m=-k$. The former case leads to a contradiction, because it implies either $y=0$ or $z=0$. If $m=-k\neq0$, the first non-zero entry of one tuple cannot be positive and thus does not lie in $G_1^+$.
 \end{remark}

We will now analyze how this binary product interacts with the construction of \Cref{prop:prescribed}.  In particular, given two totally ordered abelian groups $G$ and $H$, it is natural to ask if there is a way of constructing $R(G\oplus H)$ from $R(G)$ and $R(H)$. The following provides such a result.%

\begin{proposition}\label{prop:constructionOnRingLevel}
    Let $G$, $H$ be totally ordered abelian groups, and $k$ be a field. Let $R_{Q(H)}(G)$ denote the valuation domain with residue field $Q(H)$ and value group $G$. Then the ring homomorphism
    \begin{align*}
      \psi\colon k[t^{G\oplus H}] &\to Q_{Q(H)}(G)\\
      \sum_i c_it^{(g_i,h_i)}&\mapsto \frac{\sum_i \frac{c_it^{h_i}}{1}t^{g_i}}{1}
    \end{align*}
    induces a unique isomorphism of fields $\overline{\psi}\colon Q(G\oplus H)\to Q_{Q(H)}(G)$. Furthermore, there exists a valuation $\nu\colon Q_{Q(H)}(G)\to G\oplus H\cup\{\infty\}$ satisfying $\nu_{Q(G\oplus H)}=\nu\circ \overline{\psi}$. In particular, $R(G\oplus H)\cong R_\nu$. 
\end{proposition}

\begin{proof}
    It is clear that $\psi$ defines a ring homomorphism which preserves non-zero elements. Therefore, by appealing to the universal property of localizations of rings, we obtain $\overline{\psi}$. Let us similarly construct its inverse. Consider the well-defined map
    \begin{align*}
      \eta\colon Q(H)[t^G] &\to Q(G\oplus H)\\
      \sum_i\frac{\sum_j c_{ij}t^{h_{ij}}}{\sum_j d_{ij}t^{f_{ij}}}t^{g_i}&\mapsto \sum_i\frac{\sum_j c_{ij}t^{(g_i,h_{ij})}}{\sum_j d_{ij}t^{(0,f_{ij})}}.
    \end{align*}
    Similar to $\psi$, one can prove that $\eta$ is a morphism of rings which preserves non-zero elements; thus it induces a morphism of fields $\overline{\eta}\colon Q_{Q(H)}(G)\to Q(G\oplus H)$. As a morphism of fields, $\overline{\eta}$ is injective. For surjectivity we compute
    \begin{align}\label{eq:constructionOnRingLevel}
    \overline{\eta}\circ \overline{\psi}\left(\frac{\sum_i c_it^{(g_i,h_i)}}{\sum_i d_it^{(e_i,f_i)}}\right)=
    \overline{\eta}\left(\frac{\frac{\sum_i \frac{c_it^{h_i}}{1}t^{g_i}}{1}}{\frac{\sum_i \frac{d_it^{f_i}}{1}t^{e_i}}{1}}\right)=
    \frac{\frac{\sum_i c_it^{(g_i,h_i)}}{1}}{\frac{\sum_i d_it^{(e_i,f_i)}}{1}}=
    \frac{\sum_i c_it^{(g_i,h_i)}}{\sum_i d_it^{(e_i,f_i)}},    
    \end{align}
    i.e., $\overline{\eta}\circ \overline{\psi}=\mathrm{id}_{Q(G\oplus H)}$. Therefore, $\overline{\eta}$ is an isomorphism. Since inverses are unique, $\overline{\psi}$ must be an isomorphism too.

    Similar as in the proof of \Cref{prop:prescribed}, we will construct $\nu\colon Q_{Q(H)}(G)\to G\oplus H\cup\{\infty\}$ by exhibiting a map on $Q(H)[t^G]$ first, and then extend it to a valuation on its field of fractions, i.e., $Q_{Q(H)}(G)$. To this end, define
    \begin{align*}
        \nu\colon Q(H)[t^G]&\to G\oplus H\cup\{\infty\}\\
        \sum q_it^{g_i}&\mapsto
        \begin{cases}
            \left(g_j\coloneqq\min\limits_{q_i\neq0}g_i,\nu_{Q(H)}(q_j)\right),&\mbox{ if there is some }q_i\neq0\\
            \infty, &\mbox{ else}.
        \end{cases}
    \end{align*}
    Let us prove that $\nu$ satisfies \Cref{def:valuations}(1)-(3). Evidently, $\nu(x)=\infty$ if and only if $x=0$. Now let $x,y\neq0$, and assume without loss of generality $\nu(x)=(g_1,\nu_{Q(H)}(q_1)), \nu(y)=(\tilde g_1,\nu_{Q(H)}(\tilde q_1))$. Computing the valuation of their product yields
    \begin{align*}
        \nu(xy)&=(g_1+\tilde g_1, \nu_{Q(H)}(q_1\tilde q_1))=(g_1+\tilde g_1,\nu_{Q(H)}(q_1) + \nu_{Q(H)}(\tilde q_1))\\
        &= (g_1,\nu_{Q(H)}(q_1)) + (\tilde g_1,\nu_{Q(H)}(\tilde q_1)) = \nu(x) + \nu(y),
    \end{align*}
    as desired. Assume without loss of generality $\nu(x+y)<\infty$, and $\nu(x)\le\nu(y)$ in $G\oplus H$, i.e., one of the following cases applies:
    \begin{enumerate}
        \item\label{en:case1} $g_1< \tilde g_1$. The defining property of $g_1$ tells us that the coefficient of $t^{g_1}$ in $y$ has to be zero, i.e., this power of $t$ is not in $y$. Thus,
        \[
        \nu(x+y)=(g_1,\nu_{Q(H)}(q_1+0))=\nu(x)\ge\min(\nu(x),\nu(y)).
        \]
        \item\label{en:case2} $g_1=\tilde g_1$ and $\nu_{Q(H)}(q_1)<\nu_{Q(H)}(\tilde q_1)$. There are two possibilities in this case: Either $q_1+\tilde q_1$ is zero, or non-zero. In the former case, we have
        \[
        \nu(x+y)=(g_1,\nu_{Q(H)}(q_1+\tilde q_1))\ge (g_1,\min(\nu_{Q(H)}(q_1),\nu_{Q(H)}(\tilde q_1)))=\nu(x),
        \]
        using our assumption. In the latter case, there are $g_1=\tilde g_1<g\in G, q\in Q(H)$ such that $\nu(x+y)=(g,\nu_{Q(H)}(q))$. The existence of $g$ is due to $\nu(x+y)<\infty$. Hence, $\nu(x+y)>\nu(y)\ge\nu(x)$ by definition of the order on $G\oplus H$.
        \item $g_1=\tilde g_1$ and $\nu_{Q(H)}(q_1)=\nu_{Q(H)}(\tilde q_1)$.
        Using the assumption we compute
        \[
        \nu(x+y)=(g_1,\nu_{Q(H)}(q_1+\tilde q_1))\ge (g_1,\min(\nu_{Q(H)}(q_1),\nu_{Q(H)}(\tilde q_1)))=\nu(x).
        \]
    \end{enumerate}
    As $\nu$ satisfies \Cref{def:valuations}(1)-(3), we may extend it by defining $\nu(x/y)=\nu(x)-\nu(y)$, for every $x/y\in Q_{Q(H)}(G)$. As in the proof of \Cref{prop:prescribed}, one may verify that this $\nu$ extends to a valuation with value group $G\oplus H$.

    We are left to prove $\nu_{Q(G\oplus H)}=\nu\circ \overline{\psi}$. Let
    \[
    x/y=\frac{\sum_i c_it^{(g_i,h_i)}}{\sum_i d_it^{(e_i,f_i)}}\in Q(G\oplus H),
    \]
    and assume without loss of generality that $\nu_{Q(G\oplus H)}(x/y)=(g_1,h_1)-(e_1,f_1)$. As computed in \eqref{eq:constructionOnRingLevel}, we have 
    \[
    \overline{\psi}(x/y)=\frac{\sum_i \frac{c_it^{h_i}}{1}t^{g_i}}{\sum_i \frac{d_it^{f_i}}{1}t^{e_i}}=\frac{\sum_i p_i t^{g_i}}{\sum_i q_i t^{e_i}}.
    \]
    Since there might be $i,j$ such that $g_i=g_j$ but $h_i\neq h_j$ (similarly with $e_i,f_i$), one has to rewrite the fraction $ \overline{\psi}(x/y)$. Doing so yields the fractions $u\in Q(H)$ (resp., $v\in Q(H)$) which consist of the sum of all $p_i$ (resp., $q_i$), such that $g_i=g_1$ (resp., $e_i=e_1$). These fractions satisfy
    \[
    \nu(\overline{\psi}(x/y))=(g_1,\nu_{Q(H)}(u))-(e_1,\nu_{Q(H)}(v))=(g_1,h_1)-(e_1,f_1),
    \]
    which equals $\nu_{Q(G\oplus H)}(x/y)$. Therefore, $\nu_{Q(G\oplus H)}=\nu\circ \overline{\psi}$.
\end{proof}

We will end this section by introducing the lexicographic coproduct in more generality, which will become relevant in Section \ref{sect:generalized}.
\begin{definition}\label{def:lexicographicSupport}
    Let $I$ denote a {\it well-ordered} set, i.e., every subset of $I$ has a minimal element. Given a collection of totally ordered abelian groups $(G_i)_{i\in I}$, define the {\it support} of an element $w\in \bigoplus\limits_{i\in I}G_i$ by 
    \[
    \supp(w)=\inf\{i\in I\mid w_i\neq0\}.
    \]
    In particular, the support gives us the following description of the lexicographical order of $G=\bigoplus\limits_{i\in I}G_i$: For any two elements, $v,w\in G$, we have
    \[
v<w\iff \begin{cases}v_{\supp(v)}<w_{\supp(w)}, &\mathrm{if}\,\supp(v)=\supp(w)\\v_{\supp(v)}<0, &\mathrm{if}\, \supp(v)<\supp(w)\\w_{\supp(w)}>0, &\mathrm{if}\, \supp(v)>\supp(w).\end{cases}
\]
Moreover, note that $v,w\in G^+$ satisfy $\supp(v+w)=\min(\supp(v),\supp(w))$.
\end{definition}
\section{The construction of Bazzoni--{\v{S}}t'ov\'{\i}\v{c}ek}\label{sect:stovicek}

In this section we briefly present some results of \cite{gldm1} and put them into context with \Cref{sect:VD}. We start by defining the notion of (flat) homological epimorphisms.

\begin{definition}
    A ring homomorphism $f\in\CRings(R,S)$ is a {\it homological epimorphism} if the multiplication map $S\otimes_RS\to S$ is an isomorphism (i.e., $f$ is an epimorphism of rings) and $\Tor_i^R(S,S)=0$ for all $i\ge1$. In the particular case that $S$ is a flat $R$-module, the latter property is always satisfied, and we will call such $f$ a {\it flat epimorphism}.

    Moreover, we call two homological epimorphisms $f\in\CRings(R,S)$, $g\in\CRings(R,S^\prime)$ {\it equivalent} if there exists an isomorphism of rings $\varphi\colon S\to S^\prime$, such that $g=\varphi\circ f$. We will denote the collection of these equivalence classes by $\home(R)$. Analogously for flat epimorphisms, we denote the collection of their equivalence classes by $\flate(R)\subseteq\home(R)$.
\end{definition}

The main result of \cite{gldm1} is that there is a strong connection between smashing ideals and homological epimorphisms of rings of weak global dimension one (of which valuation domains are a particular example). In the specific case of valuation domains, these are further related to a combinatorial structure that we now introduce. We remind the reader that $\Spec(R)$ is a totally ordered poset via inclusion for $R$ a valuation domain.

\begin{definition}[{\cite[Definition 5.7]{gldm1}}]\label{def:inter}
Let $R$ be a valuation domain. An {\it admissible interval} $[\fraki,\fp]$ is an interval in $\Spec(R)$ such that $\fraki^2=\fraki\subseteq\fp$, i.e., it is the collection of prime ideals between $\fraki=\fraki^2$ and $\fp$. Denote the collection of admissible intervals by $\inter(R)$. It defines a poset by defining
\[
[\fraki,\fp]<[\fj,\fq] \mbox{ if and only if }\fp\subsetneq\fj.
\]
\end{definition}
\begin{construction}[{\cite[Construction 5.22]{gldm1}}]\label{constr:interc}
Let $R$ be a valuation domain, $(\cI,\le)$ a non-empty chain of $(\inter(R),\le)$, i.e., a subposet such that its order is induced by $\inter(R)$, and $\cI$ is totally ordered. Furthermore, assume the following properties are satisfied by $\cI$:
\begin{enumerate}[label=(C\arabic*)]
\item If $\mathcal{S}=\{[\fj_\ell,\fq_\ell]\mid \ell\in\Lambda\}$ is a non-empty subset of $\cI$ with no minimal element, then $\cI$ contains an element of the form $[\fj, \bigcap_{\ell\in\Lambda} \fq_\ell]$.
\item Dually, if $\mathcal{S}=\{[\fj_\ell,\fq_\ell]\mid \ell\in\Lambda\}$ is a non-empty subset of $\cI$ with no maximal element, then $\cI$ contains an element of the form $[\bigcup_{\ell\in\Lambda}\fj_\ell, \fq]$.
\item Given any pair $[\fj_0,\fq_0] < [\fj_1,\fq_1]$ in $\cI$, then there are elements $[\fj,\fq], [\fj^\prime,\fq^\prime]$ in $\cI$ such that
\[
[\fj_0,\fq_0]\le[\fj,\fq]<[\fj^\prime,\fq^\prime]\le [\fj_1,\fq_1],
\]
and there is no other interval in $\cI$ in between $[\fj,\fq]$ and $[\fj^\prime,\fq^\prime]$.
\end{enumerate}
We will now fix some notation, which will be important for the construction. Denote by
\[\interc(R)\coloneqq\{\cI\subset\inter(R)\mid \cI\neq\varnothing \mbox{ is a chain satisfying (C1)-(C3)}\},\]
and let $[\fraki,\fp]=\bigwedge\cI$, $[\fraki^\prime,\fn]=\bigvee\cI$; i.e., the unique minimal and maximal element of $\cI$ respectively. Note that every prime ideal of the intervals in $\cI$ lies in $[\fraki,\fn]$, since $\cI$ is totally ordered. Consider a finite disjoint union $\cI=\cI_0\cup\hdots\cup\cI_n$ of chains in $\inter(R)$ which satisfies the following conditions:
\begin{enumerate}
\item For each $0\le \ell\le n$, $\cI_\ell$ is a subchain of $\cI$ and $[\fraki_\ell,\fp_\ell]=\bigwedge\cI_\ell, [\fraki^\prime_\ell,\fn_\ell]=\bigvee\cI_\ell$ exist in $\cI_\ell$.
\item If $j< \ell$, then $[\fj,\fq]<[\fj^\prime,\fq^\prime]$ for each $[\fj,\fq]\in\cI_j$, $[\fj^\prime,\fq^\prime]\in\cI_\ell$.
\end{enumerate}
Such disjoint unions exist by iterated application of (C3). Indeed, by applying it to the minimal and maximal element of some $\cI_\ell$ such that these are distinct, yields a disjoint union of $\cI_\ell$; which in turn gives the desired presentation of $\cI$ inductively. Using this, the idea will be to consider the direct limit over $n$ in the following construction.

Let $R_{\cI} \coloneqq\prod_{[\fj,\fq]\in\cI}R_\fq/\fj$ and consider the canonical map $R_\fn\to R_{\cI}$, which maps into every factor. Then the kernel is precisely the intersection of those prime ideals $\fj$, with $[\fj,\fq]\in\cI$; this intersection is $\fraki$. Hence, this map induces the embedding $g_{\cI}\colon R_\fn/\fraki\to R_{\cI}$. Similarly, given a disjoint union $\cI=\cI_0\cup\hdots\cup\cI_n$ with notation as above, define 
\[
g_{(\cI_0,\hdots,\cI_n)}\colon \prod\limits_{\ell=0}^nR_{\fn_\ell}/\fraki_\ell\to R_{\cI}
\]
to be the composition of the product of the maps $g_{\cI_\ell}\colon R_{\fn_\ell}/\fraki_\ell\to R_{\cI_\ell}$, with the canonical isomorphism $\prod_{\ell=0}^n R_{\cI_\ell}\cong R_{\cI}$. Hence, $g_{(\cI_0,\hdots,\cI_n)}$ is again an embedding. Now the point is that the image of these embeddings forms a direct system of subrings of $R_{\cI}$ over $n$. Denote by $R_{(\cI)}$ the direct limit of these images, i.e., the union of them. Let $f_{(\cI)}\colon R\to R_{(\cI)}$ be induced by
\[
\begin{tikzcd}[row sep=.5cm,column sep=.5cm]
R\arrow[r] &R_\fn/\fraki\arrow[r] &\prod\limits_{\ell=0}^nR_{\fn_\ell}/\fraki_\ell\arrow[rrr,"g_{(\cI_0,\hdots,\cI_n)}"] & & & \mathrm{Im}(g_{(\cI_0,\hdots,\cI_n)})\subset R_{\cI},                 
\end{tikzcd}
\]
that is, $f_{({\cI})}$ is the composite $R\to R_\fn/\fraki \ontop{$g_{\cI}$}{\to}R_{\cI}$, restricted to its image which lies in $R_{(\cI)}$.
\end{construction}

\begin{theorem}[{\cite[Theorem 3.10, Theorem 6.8, Theorem 5.23, Proposition 5.5]{gldm1}}]\label{thm:bijectionsofframes}
    Let $R$ be a valuation domain. Then there are bijections
    \[
    \Sm(\sfD(R)) \cong \home(R) \cong \interc(R)\cup\{\varnothing\}.
    \]
    Under the first bijection, a class $[f\colon R\to S]\in\home(R)$ is sent to the smashing ideal $\mathrm{ker}(-\otimes^\mathbf{L}_RS)$. Given a homological epimorphism $f\colon R\to S$, let $\cI(f)$ contain all intervals $[\fj,\fq]$ such that there exists a maximal ideal $\fn\subset S$ with $\fq=f^{-1}(\fn)$ and $\fj=\mathrm{ker}(R\to S\to S_\fn)$. The other bijection assigns $[f\colon R\to S]\in\home(R)$ to $\cI=\varnothing$ if $S=0$, and $\cI(f)$ otherwise. The inverse is given by sending a non-empty chain $\cI\in\interc(R)$ to the class of $f_{({\cI})}\colon R\to R_{(\cI)}$, as in \Cref{constr:interc}.

    Moreover, these bijections restrict to the following correspondences:
    \[
    \Th(\sfD(R)^\omega) \cong  \flate(R) \cong 
\{\text{Chains }\cI\text{ of }\inter(R) \cup \{\varnothing\}\text{ such that }\cI=\{[0,\fp]\}\text{ or }\cI=\varnothing\}.
    \]
\end{theorem}

\begin{remark}
    The first bijection of \cref{thm:bijectionsofframes} is generalised to all commutative rings by Hrbek in \cite{hrbek2024telescopeconjecturehomologicalresidue}. It seems, however, that the second bijection is specific to the case of valuation domains, given their totally ordered prime ideals.
\end{remark}

\begin{remark}\label{rem:frameordering}
    The set $\Sm(\sfD(R))$ is ordered by inclusion, giving it the structure of a frame. Using this order, the bijections of \Cref{thm:bijectionsofframes} impose an ordering on $\home(R)$ and $\interc(R)\cup\{\varnothing\}$. For our purposes, we demand that the bijection between $\Sm(\sfD(R))$ and $\home(R)$ is order-preserving, while the bijection between $\home(R)$ and $\interc(R) \cup \{\varnothing \}$ is order-reversing. The reasoning for this is somewhat superficial, in that we want the empty chain to be the minimal element of the frame. However, given how this affects the meet and join in $\interc(R) \cup \{\varnothing \}$, this seems to be the natural choice (which will be discussed in \Cref{prop:meet_and_join}). Moreover, we note that $[f]\le[g]$ holds for two classes of homological epimorphisms if and only if there exists a factorization $hf=g$, where $h$ is a suitable ring homomorphism.

    By abuse of notation, we may denote the points of $\Spc^\mathrm{s}(\sfD(R))$ (i.e., the topological space associated to $\Sm(\sfD(R))$ via Stone duality) by their corresponding elements from the frame $\home(R)$ or $\interc(R) \cup \{ \varnothing \}$ under the bijections of \cref{thm:bijectionsofframes}.
\end{remark}
From now on, we will always assume that our valuation domain $R$ is finite dimensional.

Let us provide some description about the meet and join transferred from $\Sm(\sfD(R))$ to $\home(R)$ and $\interc(R)\cup\{\varnothing\}$ under the bijection of \cref{thm:bijectionsofframes}. The following lemma is immediate.
\begin{lemma}\label{lem:constructing_new_chains}
    Let $\cI,\cJ\in\interc(R)\cup\{\varnothing\}$. Then one can construct new chains as follows:
    \begin{itemize}
        \item Merging all potentially overlapping intervals in $\cI\cup\cJ$ yields a chain denoted $\cI\lor\cJ$.
        \item The pairwise intersections of all intervals in $\cI$ and $\cJ$ yields a chain denoted $\cI\wedge\cJ$.
    \end{itemize}
\end{lemma}

\begin{example}\label{ex:meets_joins}
    Suppose there is a valuation domain $R$ of Krull dimension 2 such that all prime ideals $(0)\subset \fp\subset \fm$ are idempotent (we explicitly construct such a ring in \Cref{ex:keller=2}). Then the following gives some examples of the constructions of \Cref{lem:constructing_new_chains}.
    \begin{itemize}
        \item $\{[\fp,\fm]\}\lor\{[(0),\fp]\}=\{[0,\fm]\}$ and $\{[\fp,\fm]\}\wedge\{[(0),\fp]\}=\{[\fp,\fp]\}$.
        \item $\{[\fp,\fm]\}\lor \{[(0),(0)],[\fp,\fp]\}=\{[(0),(0)],[\fp,\fm]\}$ and $\{[\fp,\fm]\}\wedge \{[(0),(0)],[\fp,\fp]\}=\{[\fp,\fp]\}$.
    \end{itemize}
\end{example}
In order to prove that the constructed chains of \Cref{lem:constructing_new_chains} define the join and meet as suggestively written, we require the following lemma. Its proof is straightforward, but it is indispensable for determining factorizations between homological epimorphisms in terms of the corresponding chains.
\begin{lemma}\label{lem:contracting}
    Let $R$ be a valuation domain and let $[\fraki_1,\fp_1],[\fraki_2,\fp_2]$ be admissible intervals in $R$. Then the natural map $R_{\fp_1}/\fraki_1\to R_{\fp_2}/\fraki_2$ exists if and only if $\fraki_1\subseteq\fraki_2\subseteq\fp_2\subseteq\fp_1$.
\end{lemma}
\begin{proposition}\label{prop:meet_and_join}
    Let $[f],[g]\in\home(R)$. Then the join $[f]\lor[g]$ is represented by the homological epimorphism corresponding to $\cI(f)\wedge\cI(g)$. Dually, the meet $[f]\wedge[g]$ is represented by the homological epimorphism corresponding to $\cI(f)\lor\cI(g)$. In particular, the meet and join in $\interc(R)\cup\{\varnothing\}$ are given by the constructions in \Cref{lem:constructing_new_chains}, as suggestively denoted.
\end{proposition}
\begin{proof}
    Let $h$ be the homological epimorphism constructed from $\cI(f)\wedge\cI(g)$ using \Cref{constr:interc}. Using \Cref{lem:contracting}, it is clear that $[h]$ is indeed the join of $[f]$ and $[g]$, since the factorizations $[f],[g]\le[h]$ contract every interval in $\cI(f),\cI(g)$ to the maximal common one, i.e., the intervals in $\cI(f)\wedge\cI(g)$, yielding factorizations which turn $[h]$ into the join of $[f]$ and $[g]$.

    \indent
    Dually, let $h$ be the homological epimorphism constructed from $\cI(f)\lor\cI(g)$. Evidently, we have $[h]\le[f],[g]$ by \Cref{lem:contracting}. Assume there exists some $[\tilde h]\le [f],[g]$. In other words, each interval in $\cI(f),\cI(g)$ is necessarily covered by some interval in $\cI(\tilde h)$, because of \Cref{lem:contracting}. As $\cI(h)$ contains all minimal such coverings by definition, there must be a factorization $[\tilde h]\le [h]$.
\end{proof}

\begin{remark}
    We highlight that under our construction, the meet-prime elements of $\Sm(\sfD(R))$ and $\home(R)$ correspond to the join-prime elements of $\interc(R) \cup \{\varnothing\}$. We will freely use this fact without further warning.
\end{remark}

From \Cref{thm:bijectionsofframes} we immediately deduce the following observation.
\begin{corollary}\label{prop:VD&TC}
Let $R$ be a valuation domain. Then $\sfD(R)$ does not satisfy the Telescope Conjecture if and only if there exists some prime ideal $0\neq\fp\subset R$ such that $\fp^2=\fp$. Each such idempotent prime ideal gives rise to at least one unique smashing ideal which is not inflated from a compactly generated thick ideal.
\end{corollary}

Recalling \Cref{prop:idealsAndFilters}, \Cref{prop:VD&TC} can be rephrased in terms of the valuation group associated to the valuation ring. That is, for a valuation domain $R$ with value group $G$, $\sfD(R)$ satisfies the telescope conjecture if and only if there is no idempotent prime filter in $G^+$. This insight turns out to be valuable in order to determine whether the telescope conjecture holds for $\sfD(R)$. Let us give a motivating example for this:

\begin{example}[{\cite[Example 5.24]{gldm1}}, {\cite[\S 7.2]{balchin2023big}}, Keller's example]\label{ex:keller}
Let $\ell\ge2$, and consider the totally ordered abelian group $G=G_\ell=(\Z\left[\frac{1}{\ell}\right],\le, +)$, i.e., the localization of $\Z$ at powers of the fixed $\ell\ge2$, with order induced by their fractions in $\Q$. That is, for the multiplicative set $S=\{1,\ell,\ell^2,\hdots\}$ we have $S^{-1}\Z=\Z\left[\frac{1}{\ell}\right]$, and the natural map $S^{-1}\Z\to\Z_{(0)}\cong\Q$ induces a total order on $S^{-1}\Z$.

    Fix some field $k$. Then \Cref{prop:prescribed} yields the valuation domain $A = R_k\left(\Z\left[\frac{1}{\ell}\right]\right)$ which coincides with the ring also denoted by $A$ in \cite[\S2]{keller94}. That is, $A$ is the localization of $k[t,t^{1/\ell},t^{1/\ell^2},\hdots]$ at the maximal ideal $\fm=(t,t^{1/\ell},t^{1/\ell^2},\hdots)$. Let us determine the Zariski spectrum of $A$. 
    Using \Cref{prop:idealsAndFilters}, the non-zero prime ideals of $A$ correspond to prime filters in $G^+$. We claim that every prime filter $F\subset G^+$ corresponds to $\fm$. By uniqueness of the maximal ideal, it suffices to prove $\nu(\fm)\subseteq F$. For the sake of contradiction, suppose there is some $1/\ell^k$ for some $k\ge1$ such that $1/\ell^k\not\in F$, i.e., $1/\ell^k\in G^+\setminus F$. Since $F$ is a prime filter, any multiples of $1/\ell^k$ also lie in $G^+\setminus F$. This, however, contradicts the upwards closure of $F$, as such product is unbounded. Hence $\Spec(A)=\{(0),\fm\}$. Additionally, $\fm^2=\fm$. 
    
    Applying \Cref{constr:interc}, we find
    \[
    \interc(A)=\{\{[(0),(0)]\}, \{[(0),\fm]\},\{[\fm,\fm]\},\{[(0),(0)], [\fm,\fm]\}\}.
    \]
    Let $Q=A_{(0)}$. Combining this description of $\interc(A)$ with \Cref{thm:bijectionsofframes}, we can summarize the resulting equivalence classes of homological epimorphisms with domain $A$, and the smashing frame of $\sfD(A)$, as follows:
    
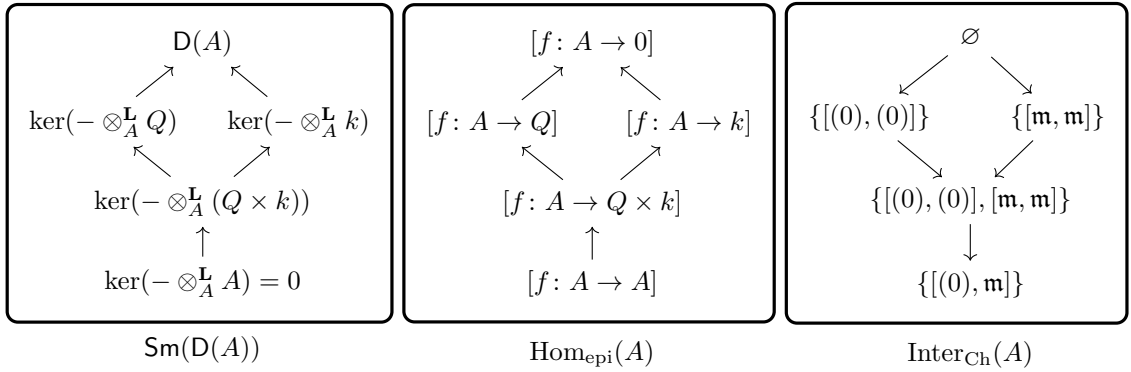
\[    
\begin{tikzcd}[mymatr, row sep=0.41cm,column sep=-1.4cm]
                                         & \sfD(A) \arrow[rd, leftarrow] \arrow[ld, leftarrow] &                                         \\
\mathrm{ker}(-\dert_A Q) \arrow[rd, leftarrow] &                                                 & \mathrm{ker}(-\dert_Ak) \arrow[ld, leftarrow] \\
                                         & \mathrm{ker}(-\dert_A(Q\times k)) \arrow[d, leftarrow]  &                                         \\
                                         & \mathrm{ker}(-\dert_A A)=0   \ar[phantom, "\Sm(\sfD(A))"below=18pt]                       &                                        
\end{tikzcd}
\,\,
\begin{tikzcd}[mymatr,row sep=0.45cm,column sep=-1.05cm]
                       & {[f\colon A\to 0]} \arrow[ld, leftarrow] \arrow[rd, leftarrow] &                                                         \\
                    {[f\colon A\to Q]} \arrow[rd, leftarrow] &                                                            & {[f\colon A\to k]} \arrow[ld, leftarrow]                                                       \\
                                                           & {[f\colon A\to Q\times k]} \arrow[d, leftarrow]              &                                                                                              \\
                                                          & {[f\colon A\to A]}     \ar[phantom, "\home(A)"below=19pt]                                    &                                                                                             
\end{tikzcd}
\,\,
\begin{tikzcd}[mymatr,row sep=0.5cm,column sep=-1.1cm]
                       & {\varnothing} \arrow[ld] \arrow[rd] &                                                         \\
                    {\{[(0),(0)]\}} \arrow[rd] &                                                            & {\{[\fm,\fm]\}} \arrow[ld]                                                       \\
                                                           & {\{[(0),(0)],[\fm,\fm]\}} \arrow[d]              &                                                                                              \\
                                                          & {\{[(0),\fm]\}}    \ar[phantom, "\interc(A)"below=19pt]                                    &                                                                                             
\end{tikzcd}
\]

In particular, \Cref{thm:bijectionsofframes} tells us that the only compactly generated smashing ideals correspond to the representatives $A\to 0$, $A\to Q$, and $A\to A$, which are flat.

As $\Sm(\sfD(A))$ is spatial, we can appeal to Stone duality to obtain the smashing spectrum and compare it to the (Hochster dual of the) Balmer spectrum. In the notation of $\interc(A)$ we have:
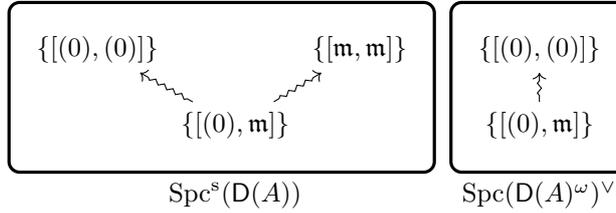
\[
\begin{tikzcd}[style={every outer matrix/.append style={draw=black, inner xsep=6pt , inner ysep=9pt, rounded corners, very thick}},row sep=0.4cm,column sep=0.0cm]
\{[(0),(0)]\} && \{[\fm,\fm]\} \\
& \{[(0),\fm]\} \arrow[ru, rightsquigarrow] \arrow[lu, rightsquigarrow]  \ar[phantom, "\Spc^\mathrm{s}(\sfD(A))"below=19pt] \end{tikzcd}
\,\,\,
\begin{tikzcd}[style={every outer matrix/.append style={draw=black, inner xsep=6pt , inner ysep=9pt, rounded corners, very thick}},row sep=0.4cm,column sep=0.4cm]
\{[(0),(0)]\} \\ \{[(0),\fm]\}  \arrow[u, rightsquigarrow] \ar[phantom, "\Spc(\sfD(A)^\omega)^\vee"below=19pt]
\end{tikzcd}
\]
As usual, an arrow denotes specialization closure, so that the bottom points are open while the top points are closed.

The comparison map $\Spc^\mathrm{s}(\sfD(A)) \to \Spc(\sfD(A)^\omega)^\vee$ is characterized as
\[
\{[(0),(0)]\} \mapsto \{[(0),(0)]\}  \qquad
\{[(0),\fm]\} \mapsto \{[(0),\fm]\} \qquad
\{[\fm,\fm]\} \mapsto \{[(0),\fm]\}.
\]

\end{example}

\section{Generalising Keller's example}\label{sect:generalized}

In \cref{ex:keller} we investigated, in depth, the example of Keller's ring $A$ from the point of view of \cref{thm:bijectionsofframes}. In this section, we will use the constructions of \cref{sect:VD} to generalise this example. We begin by explicitly discussing the $n=2$ example.

\subsection{Keller's example generalised, $\mathbf{n=2}$}\label{ex:keller=2}
Let $\ell_1,\ell_2\ge2$ and consider the totally ordered abelian group $G =\Z\left[\frac{1}{\ell_1}\right]\oplus \Z\left[\frac{1}{\ell_2}\right]$, where the group structure of each summand has been discussed in \Cref{ex:keller}, and the ordering is defined by the lexicographical order. For a fixed field $k$, let us denote $R_k(G)$ by $A^{\star 2}$. To determine $\Spec(A^{\star 2})$, we will use \Cref{prop:idealsAndFilters} and classify the (idempotent) prime filters of $G^+$.

Let $F\subsetneq G^+$ be a non-empty subset, and let $w\in G^+\setminus F$ be any element satisfying the following property: For every $v\in G^+\setminus F$, the index of the first non-zero entry of $w$ is at most the index of the first non-zero entry of $v$. In other words, among the elements in $G^+\setminus F$, the first non-zero entry of $w$ has a minimal index. If $F$ is a prime filter, then any multiples of $w$ do not lie in $F$ and upwards closure implies that $w$ has to be of the form $(0,q)$. Thus, we end up with
\[
G^+\setminus F=\begin{cases}
     \{(s,t)\in G^+\mid s=0, t\neq 0\}, &\mbox{ if }w=(0,q),\, q\neq0\\
      \{(s,t)\in G^+\mid s=t=0\}, &\mbox{ if }w=(0,0).
\end{cases}
\]
Hence, the prime filters must be of the form $F=\{(s,t)\in G^+\mid s\neq0, t=0\}$ or $F=G^+\setminus\{0\}$. We claim that both filters are idempotent. As this may be verified in each entry separately, we will show that the former is idempotent. Since $F\subset G^+$, we know that every element in $F$ is of the form $(n/\ell_1^k,0)$, for some $n\ge 1,k\ge0$. Using that $\ell_1\ge 2$, we have for all $n\ge 2$, $k\ge1$:
\begin{equation}\label{eq:thickening}
    \frac{1}{1}=\frac{\ell_1-1}{\ell_1}+\frac{1}{\ell_1},\quad \frac{n}{1}=\frac{n-1}{1}+\frac{1}{1},\quad\frac{1}{\ell_1^k}=\frac{\ell_1-1}{\ell_1^{k+1}}+\frac{1}{\ell_1^{k+1}}, \quad\frac{n}{\ell_1^k}=\frac{n-1}{\ell_1^k}+\frac{1}{\ell_1^k}.
\end{equation}
Using \Cref{prop:idealsAndFilters}, we deduce that $A^{\star 2}$ contains precisely three prime ideals, each of which is idempotent. Denote them by $(0)\subset \fp\subset \fm$. Let us denote $Q=A^{\star 2}_{(0)}$, and $k(\fp)=A^{\star 2}_\fp/\fp$. Using \Cref{thm:bijectionsofframes} and applying \Cref{constr:interc} to the valuation domain $A^{\star 2}$, we can summarize all elements in $\interc(A^{\star 2})$ and $\home(A^{\star 2})$ in \Cref{tab:n=2}.
\begin{center}
\begin{tabular}{||P{.43\textwidth}||P{.52\textwidth}||}
\hline
\begin{tabular}{P{.14\textwidth}|P{.26\textwidth}}
$\cI$ &Repr.\ in $\home(A^{\star 2})$
\end{tabular}
&
\begin{tabular}{P{.25\textwidth}|P{.25\textwidth}}
$\cI$ & Repr.\ in $\home(A^{\star 2})$
\end{tabular}
\\
\hline
\begin{tabular}{P{.14\textwidth}|P{.26\textwidth}}
$\varnothing$ & $A^{\star 2}\to 0$\\
$\{[(0),(0)]\}$ & $A^{\star 2}\to Q$\\
$\{[(0),\fp]\}$ & $A^{\star 2}\to A^{\star 2}_\fp$\\
$\{[(0),\fm]\}$ & $A^{\star 2}\to A^{\star 2}$\\
$\{[\fp,\fp]\}$ & $A^{\star 2}\to k(\fp)$\\
$\{[\fp,\fm]\}$ & $A^{\star 2}\to A^{\star 2}/\fp$\\
$\{[\fm,\fm]\}$ &  $A^{\star 2}\to k$
\end{tabular}
&
\begin{tabular}{P{.25\textwidth}|P{.25\textwidth}}
 $\{[(0),(0)],[\fp,\fp]\}$ &  $A^{\star 2}\to Q\times k(\fp)$\\
 $\{[(0),(0)],[\fm,\fm]\}$ &  $A^{\star 2}\to Q\times k$\\
 $\{[\fp,\fp],[\fm,\fm]\}$ &  $A^{\star 2}\to k(\fp)\times k$\\
 $\{[(0),(0)],[\fp,\fm]\}$ &  $A^{\star 2}\to Q\times A^{\star 2}/\fp$\\
 $\{[(0),\fp],[\fm,\fm]\}$ &  $A^{\star 2}\to A^{\star 2}_\fp\times k$\\
 $\{[(0),(0)],[\fp,\fp],[\fm,\fm]\}$ &  $A^{\star 2}\to Q\times k(\fp)\times k$\\
&
\end{tabular}\\
\hline  
\end{tabular}
\captionof{table}{The empty chain along with the 12 elements of $\interc(A^{\star 2})$, and their representative in $\home(A^{\star 2})$.}\label{tab:n=2}
\end{center}

Let $[f\colon A^{\star 2}\to S]$ denote a class of a homological epimorphism, e.g., represented by one as given in \Cref{tab:n=2}. Instead of writing $\mathrm{ker}(-\dert_{A^{\star 2}} S)$ for the associated smashing ideal by \Cref{thm:bijectionsofframes}, we shall instead denote  it with the codomain of $f$, i.e., $S$. In particular, $A^{\star 2}$ sits at the bottom since $\mathrm{ker}(-\dert_{A^{\star 2}} A^{\star 2})=0$. With this notation in mind, we can describe the frame $\Sm(\sfD(A^{\star 2}))$ in the form of a Hasse diagram, as seen in \Cref{fig:n=2}. By \Cref{thm:bijectionsofframes}, the compactly generated smashing ideals correspond to $0$, $Q$, $A^{\star 2}_\fp$, and $A^{\star 2}$, which are circled.
\pgfkeys{/csteps/inner ysep=8pt}
\pgfkeys{/csteps/inner xsep=8pt}
\begin{figure}[h]
\begin{center}

\begin{tikzcd}[style={every outer matrix/.append style={draw=black, inner xsep=6pt , inner ysep=9pt, rounded corners, very thick}},row sep=0.65cm,column sep=0.2cm]
                             &                                                           & \Circled{0}    &                                                           &                              \\
                             & {\color{purple}k} \arrow[ru]                  & {\color{purple}k(\fp)} \arrow[u]                 & \Circled[inner color=purple]{Q} \arrow[lu]                 &                              \\
                             & k(\fp)\times k \arrow[u] \arrow[ru] & Q\times k \arrow[lu] \arrow[ru]                                     & Q\times k(\fp) \arrow[lu] \arrow[u] &                              \\
{\color{purple}A^{\star 2}/\fp}\arrow[ru]&                                                           & Q\times k(\fp)\times k  \arrow[lu] \arrow[ru] \arrow[u]  &                                                           & \Circled[inner color=purple]{A^{\star 2}_\fp} \arrow[lu] \\
                             & Q\times A^{\star 2}/\fp  \arrow[lu] \arrow[ru]                     &                                                                   & A^{\star 2}_\fp\times k \arrow[lu] \arrow[ru]                     &                              \\
                             &                                                           & \Circled{A^{\star 2}}  \arrow[lu] \arrow[ru]                                                               &                                                           &                             
\end{tikzcd}
\begin{center}\caption{The frame $\Sm(\sfD(A^{\star 2}))$ in terms of the codomain of the corresponding homological epimorphism. Circles elements correspond to compactly generated smashing ideals, and the elements in magenta are the meet prime elements.}\label{fig:n=2}\end{center}
\end{center}
\end{figure}
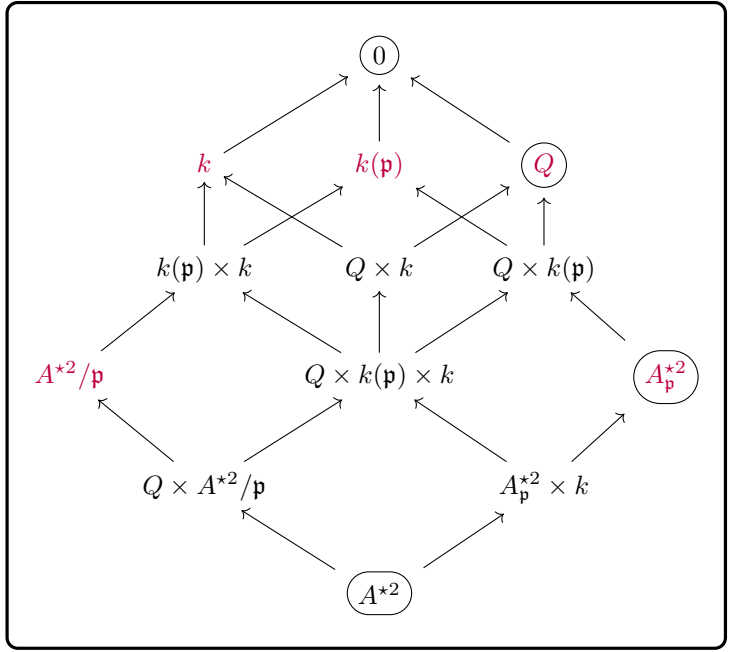

 The meet-primes of this poset can be obtained by inspection, and consist of $A^{\star 2}/\fp$, $k$, $k(\fp)$, $Q$, and $A^{\star 2}_\fp$; which are highlighted in magenta in \cref{fig:n=2}. Hence, the smashing spectrum of $\sfD(A^{\star 2})$ is given as in \cref{fig:n=2smashepi}. Additionally, we recall \Cref{ex:meets_joins}, where we computed some joins and meets. These examples are also visualized in \Cref{fig:n=2}, since $A^{\star 2}$ satisfies the properties of the ring in \Cref{ex:meets_joins}.

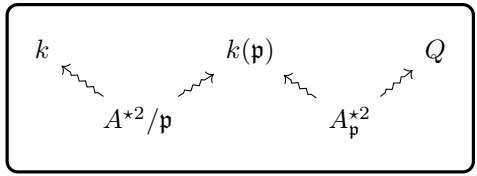
\begin{figure}[h]
\begin{center}
\begin{tikzcd}[style={every outer matrix/.append style={draw=black, inner xsep=6pt , inner ysep=9pt, rounded corners, very thick}},row sep=0.3cm,column sep=0.4cm]
k && k(\fp) && Q \\
& A^{\star 2}/\fp \arrow[ru, rightsquigarrow] \arrow[lu, rightsquigarrow] && \arrow[ru, rightsquigarrow] \arrow[lu, rightsquigarrow]  A^{\star 2}_{\fp} \end{tikzcd}
\begin{center}\caption{The smashing spectrum of $\sfD(A^{\star 2})$. The points at the top are closed, while the points at the bottom are open. An arrow denotes specialization closure.}\label{fig:n=2smashepi}\end{center}
\end{center}
\end{figure} 

At this point, it is worth to display this space with the chain representatives instead (\cref{fig:n=2smash}). When using such, a pattern seems to emerge: Its characteristics in the general case will be discussed in \cref{ex:keller=n}.

\begin{figure}[h]
\begin{center}
\begin{tikzcd}[style={every outer matrix/.append style={draw=black, inner xsep=6pt , inner ysep=9pt, rounded corners, very thick}},row sep=0.3cm,column sep=0.4cm]
\{[\fm,\fm]\} && \{[\fp,\fp]\} && \{[(0),(0)]\} \\
& \{[\fp,\fm]\} \arrow[ru, rightsquigarrow] \arrow[lu, rightsquigarrow] && \arrow[ru, rightsquigarrow] \arrow[lu, rightsquigarrow]  \{[(0),\fp]\} \end{tikzcd}
\begin{center}\caption{The smashing spectrum of $\sfD(A^{\star 2})$ in the notation of $\interc(A^{\star 2})$.}\label{fig:n=2smash}\end{center}
\end{center}
\end{figure}
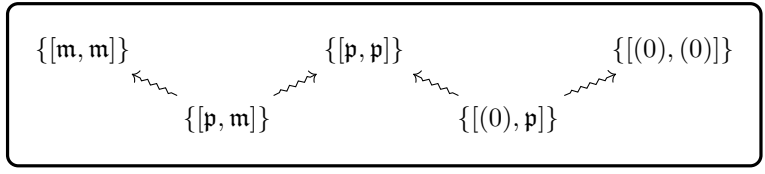

In terms of chains, the comparison map $\Spc^\mathrm{s}(\sfD(A^{\star 2})) \to \Spc(\sfD(A^{\star 2})^\omega)^\vee$ is determined by $\{[\mathfrak{a},\mathfrak{b}]\} \mapsto \{[(0), \mathfrak{b}]\}$.

\subsection{Keller's example generalised}\label{ex:keller=n}

Now that we have constructed the $n=2$ case, we will investigate the properties of the case for general $n$. For $n \geq 1$, we define $A^{\star n}$ to be the valuation domain with value group $\bigoplus_n \Z\left[\frac{1}{\ell}\right]$, with lexicographic order as described in \Cref{def:lexicographicSupport}, and residue field $k$. We begin by exploring the structure of the prime ideals of $A^{\star n}$.

\begin{proposition}\label{prop:structureofprimes}
Let $n \geq 1$. Then all prime ideals of $A^{\star n}$ are idempotent and assemble into a chain
\[
    (0)=\fp_0\subset\fp_1\subset\hdots\subset \fp_n =\fm.
   \]
\end{proposition}

\begin{proof}
 To prove the claimed statement, we will classify the prime filters in $G^+$ which is sufficient by \Cref{prop:idealsAndFilters}.  Consider any non-empty subset $F\subsetneq G^+$. Let $w\in G^+\setminus F$ be an element with minimal support, i.e.,  $\supp(w)\le\supp(v)$ for all $v\in G^+\setminus F$. This exists because every subset of $\N$ has a minimum and $F\neq G^+$. If $F$ is upwards closed, this forces
    \[
    \{v\in G^+\mid \supp(w)\le \supp(v)\}\subseteq G^+\setminus F.
    \] 
    In fact, the defining property of $w$ implies that this is an equality. Passing to the complement, we have $F=\{v\in G^+\mid \supp(v)<\supp(w)\}$. This  characterizes the prime filters in $G^+$, as for every index $1\le j\le n$ we find the prime filters $F_j\coloneq\{v\in G^+\mid \supp(v)<j+1\}$. 
    
    As for idempotence, it suffices to show that the prime filters $F_j$ are idempotent for all $1\le j\le n$. To this end, we must show that every element $v\in F_j$ can be written as the sum of two elements in $F_j$. We know each $v\in F_j$ has support $\le j$. That is, its first non-zero entry is positive and has index $\le j$. All higher entries are arbitrary, but are only finitely often non-zero. Hence, we can project $v$ onto its first $j$ entries $(v_1,\hdots,v_j)$, and reduce to the case where we show the following:
    \[
    \mbox{Every }\mkern-6mu(v_1,\hdots,v_j){\in}\left(\bigoplus\limits_{k=1}^j\Z\left[\frac{1}{\ell}\right]\right)_{>0}\mkern-24mu\mbox{ can be written as the sum of two elements in }\mkern-6mu\left(\bigoplus\limits_{k=1}^j\Z\left[\frac{1}{\ell}\right]\right)_{>0}\mkern-18mu.
    \]
    This can be achieved by picking suitable elements component-wise, depending on the sign, as done explicitly in \Cref{ex:keller=2}, see \eqref{eq:thickening}. Hence, all $F_j$ are idempotent as desired.
\end{proof}

From \cref{prop:structureofprimes} we immediately obtain all information about the Balmer spectrum of $\sfD(A^{\star n})$, and hence by Stone duality a description of $\Th(\sfD(A^{\star n})^\omega)$:

\begin{corollary}
Let $n \geq 1$. Then $\Spc(\sfD(A^{\star n})^\omega)^\vee \cong [n+1]$ equipped with the Alexandroff topology and $|\Th(\sfD(A^{\star n})^\omega)| = n+2$.
\end{corollary}

We now move towards a full description of the frame $\Sm(\sfD(A^{\star n}))$. Using \cref{constr:interc}, we start with identifying the elements of this frame by describing $\interc(A^{\star n})$. In this finite setting where all prime ideals are idempotent, the rules degenerate into a particularly simple form:

\begin{lemma}\label{lem:intervals_keller=n}
Let $n \geq 1$. Then every $\mathcal{I} \in \interc(A^{\star n})$ is of the form
\[
\mathcal{I} = \{[\fp_{i_1},\fp_{j_1}], \dots , [\fp_{i_k},\fp_{j_k}]\},
\]
where $0 \le k \le n$ satisfies
\[
\fp_{i_t} \subseteq \fp_{j_t}  \text{ for } 1 \leq t \leq k \quad \text{ and } \quad \fp_{j_t} \subsetneq \fp_{i_{t+1}}  \text{ for } 1 \leq t < k. 
\]
\end{lemma}

\begin{corollary}\label{cor:counting}
Let $n \geq 1$. Then $|\Sm(\sfD(A^{\star n}))| = \mathrm{Fib}_{2n+3}$ where $\mathrm{Fib}_i$ is the $i^\mathrm{th}$ Fibonacci number.
\end{corollary}

 \begin{proof}
        We will make use of the equality $\mathrm{Fib}_{2n+3} = \sum_{k=0}^{n+1} \binom{n+k+1}{2k}$. For $0 \leq k \leq n+1$, denote by $\interc^k(A^{\star n})$ the subset of  $\interc(A^{\star n})$, which consists of those chains with length $k$. For simplicity, denote an interval $[\fp_m,\fp_k]$ by $[m,k]$. We will prove that $|\interc^k(A^{\star n})| = \binom{n+k+1}{2k}$, whence the desired result.

        Evidently, when $k=0$ we only have the empty chain $\varnothing$, and we indeed have $\binom{n+1}{0} = 1$. If $k=1$, we have $\interc^1(A^{\star n})$ is described by the collection of pairs $[i,j]$ such that $0 \leq i \leq j \leq n$. This is well known to be counted by the triangular numbers; $\binom{n+2}{2}$.

        Let us also explicitly consider the case of $k=2$. $\interc^2(A^{\star n})$ consists of tuples $\{[i_1,j_1], [i_2, j_2]\}$ such that $0 \leq i_1 \leq j_1 < i_2 \leq j_2 \leq n$.  As we are allowed two equalities (namely $i_1 \leq j_1$ and $i_2 \leq j_2$) we have $n+3$ options to pick from, and we pick 4 objects, hence the $\binom{n+3}{4}$.

        In general, one observes that indeed $|\interc^k(A^{\star n})| = \binom{n+k+1}{2k}$, as we are considering tuples $\{[i_1,j_1], \dots , [i_k,j_k]\}$ with $i_t \leq j_t$ and $j_t < i_{t+1}$.\
    \end{proof}

\begin{table}[h]
\begin{tabular}{r|c|c|c|c|c|c}
$n$ & 1 & 2 & 3 & 4 & 5 & 6 \\ \hline
$|\Sm(\sfD(A^{\star n}))|$ & 5 & 13 & 34 & 89 & 233 & 610
\end{tabular}
\caption{A table of $|\Sm(\sfD(A^{\star n}))|$ for small values of $n$.}
\end{table}

Now that we understand the objects of $\Sm(A^{\star n})$, we will identify the frame structure which allows us to describe the smashing spectrum and the comparison map. We will do this by considering the minimal inclusions of the elements in $\interc(A^{\star n})$, which describes the corresponding  Hasse diagram.

Again, let us denote the intervals of prime ideals $[\fp_m,\fp_k]$ by $[m,k]$. Given two homological epimorphisms $f\colon A^{\star n}\to S$, $g\colon A^{\star n}\to S^\prime$, recall that these yield chains $\cI(f)$, $\cI(g)$ by \Cref{thm:bijectionsofframes}. Thus, without loss of generality, we may assume that $S$ and $S^\prime$ are products as constructed in \Cref{constr:interc}. We claim that the existence of a minimal factorization between $[f]\neq[g]$ is characterized by the following two rules:
    \begin{itemize}
        \item\textbf{Subset rule}: If $\cI(f)\subseteq \cI(g)$, then there is minimal factorization  $[g]\le[f]$ if and only if there is some $[i,i]\not\in\cI(f)$ such that $\cI(g)=\cI(f)\cup\{[i,i]\}$. In particular, if for some $[i,i]\not\in\cI(f)$ the set $\{[i,i]\}\cup\cI(f)$ is a chain, then there is a minimal factorization between $f$ and the homological epimorphism corresponding to the chain $\cI(f)\cup\{[i,i]\}$. 
        \item\textbf{Gluing rule}: If $\cI(f)\not\subseteq \cI(g)$, there is minimal factorization  $[g]\le[f]$ if and only if for some $0\le m\le k_1< k\le n$ there are $[m,k_1],[k_1+1,k]\in\cI(f)$ allowing us to write $\cI(g)=\{[m,k]\}\cup\cI(f)\setminus\{[m,k_1],[k_1+1,k]\}$. In particular, if for some $m\le k_1< k$ there are $[m,k_1],[k_1+1,k]\in\cI(f)$, then there is a minimal factorization between $f$ and the homological epimorphism corresponding to the chain $\{[m,k]\}\cup\cI(f)\setminus\{[m,k_1],[k_1+1,k]\}$.
  \end{itemize}   
   
  \begin{proposition}\label{prop:minimal_fact}
  Let $n \geq 1$ and $[f] \neq [g]$ as above. Then any minimal factorization $[g] \leq [f]$ arises either from the subset rule or the gluing rule.
  \end{proposition}

\begin{proof}
   Let us start proving the subset rule. If $\cI(f)\subseteq \cI(g)$, there is a canonical morphism $p\colon S^\prime \to S$ yielding a factorization of $f$ and $g$, such that $[g]\le[f]$. Evidently, such factorization cannot be minimal unless $\cI(f)$ differs $\cI(g)$ by one element: Assume there is a disjoint union $\cI(f)\cup\{[m,k]\}=\cI(g)$ for some $0\le m\le k\le n$. We claim that then the factorization witnessed by $p$ is minimal if and only if $m=k$, i.e., $S\times k(\fp_k)=S^\prime$. Assume $m\neq k$. Then the factorization given by $p$ is not minimal, since \Cref{lem:contracting} allows us to factorize $p$ as follows:
    \[
\begin{tikzcd}
S^\prime=S\times A^{\star n}_{\fp_k}/\fp_m \arrow[rr] \arrow[rrrr, "p", bend left=17] &  & S\times k(\fp_m) \arrow[rr] &  & S.
\end{tikzcd}
    \]
    For the other direction assume $m=k$, and in order to prove $p$ being a minimal factorization, assume $[g]\le [h]\le [f]$, i.e., there exists a commutative diagram as follows:
\begin{eqnarray}\label{eqn:cDiagramGeneralizedKellerProof}
\begin{tikzcd}
S\times k(\fp_m) \arrow[r, "p^\prime"'] \arrow[rr, "p", bend left] & T \arrow[r, "p^{\prime\prime}"']                     & S \\
                                                                   & A^{\star n}. \arrow[ru, "f"'] \arrow[lu, "g"] \arrow[u, "h"] &  
\end{tikzcd}
\end{eqnarray}
    If $T=0$, then $S=0$; in which case the factorization is indeed minimal. Otherwise, \Cref{constr:interc} tells us that $S$ and $T$ can be written as the product of rings of the form $A^{\star n}_{\fp_j}/\fp_i$, and that the morphisms are products of natural maps into these factors. More importantly, the commutativity of the diagram in \eqref{eqn:cDiagramGeneralizedKellerProof} forces that $p$ and $p^{\prime\prime}$ map to the same factors. Since $p$ maps to every factor, so must $p^{\prime\prime}$. Additionally,  $p^\prime$ has to map to every factor of $T$, since $h$ does and due to the commutativity of \eqref{eqn:cDiagramGeneralizedKellerProof}. Restricting $p$ to $S$ shows that the restriction of $p^\prime$ (resp., $p^{\prime\prime}$) is injective (resp., surjective). Hence, $T$ must contain $S$ as a subring. Assume $T$ consists of at least another factor, or in other words, $T=S\times P=S\times \prod_\alpha A^{\star n}_{\fp_{j_\alpha}}/\fp_{i_\alpha}$ and $p^\prime=(\mathrm{id}_S,\phi)$, where $\phi\colon S\times k(\fp_m)\to P$. The universal property gives us projection maps $\phi_S,\phi_{k(\fp_m)}$. Without loss of generality, it suffices to assume $\phi_S=0$ (since $p^\prime=(\mathrm{id}_S,\phi_S+\mathrm{id}_P)\circ (\mathrm{id}_S,\phi_{k(\fp_m)})$). Hence, $\phi\colon k(\fp_m)\to P$ gives rise to maps $\phi_\alpha$ which map into each factor of $P$ by the universal property of the product. By \Cref{lem:contracting}, each factor of $P$ has to be of the form $k(\fp_m)$. Since $T$ cannot contain duplicate factors, this implies $P=k(\fp_m)$ as desired, i.e., $[g]=[h]$.

    In order to prove the gluing rule, assume that $\cI(f)\not\subseteq\cI(g)$ and $[g]\le[f]$, i.e., there exists a factorization $p\circ g=f$. Then $\cI(f)\supseteq\cI(g)$ cannot hold, since the subset rule would tell us $[g]=[f]$. So there are chains $[m,k]\in\cI(g),[m_1,k_1]\in\cI(f)$, which do not lie in $\cI(f)\cap\cI(g)$. Letting $\tilde S=\prod_{[\fraki,\fp]\in \cI(f)\cap\cI(g)}A^{\star n}_{\fp}/\fraki$, we can split off $S^\prime$ and $S$ into their common factors to obtain the following commutative diagram for non-trivial products $P,Q$:
    \begin{eqnarray}\label{eqn:cDiagramGeneralizedKellerProof2}
    \begin{tikzcd}
        S^\prime=\tilde S\times P \arrow[rr, "p"] &                                              &  S=\tilde     S\times Q \\
                                          & A^{\star n} \arrow[lu, "g"] \arrow[ru, "f"'] &                   
    \end{tikzcd}
    \end{eqnarray}
    Let us now assume that $p$ is a minimal factorization. Then $p$ restricted to the factor $\tilde S$ can not map into any factor of $Q$, since this would yield a smaller factorization. Hence, $p=(\mathrm{id}_{\tilde S},\phi)$, where $\phi\colon \tilde S\times P\to Q$. Yet again, we can assume without loss of generality that $\phi\colon P\to Q$, as done in \eqref{eqn:cDiagramGeneralizedKellerProof}. We claim that $P$ consists of one factor, and $Q$ of precisely two.

    \indent
    If $P$ contained two or more factors, let us say $P_1,P_2$, we know that they are not factors in $Q$. However, they map to some factor in $Q$, or none. If any, let us say $P_1$, does not map to any factor of $Q$ via $\phi$, $p$ cannot be minimal: $T$ containing all factors of $P$ besides $P_1$ factorizes $p\colon\tilde S\times P\to \tilde S\times T\to \tilde S\times Q$. So $P_1,P_2$ map to $Q$ non-trivially. Hence, they must be mapped to different factors of $Q$ via $\phi$. A similar factorization as done with $T$ shows that $p$ cannot be minimal. So we must have $P=A^{\star n}_{\fp_{k}}/\fp_{m}$. As for $Q$, assume it contains three or more factors, i.e., there are different $[m_1,k_1],[m_2,k_2],[m_3,k_3]\in \cI(f)\setminus \cI(g)$. Without loss of generality, they are already linearly ordered in $\inter(A^{\star n})$. Since $P$ maps to every factor which corresponds to these intervals, \Cref{lem:contracting} implies $m\le m_1\le k_1<m_2\le k_2<m_3\le k_3\le k$.
    By \Cref{lem:contracting}, we obtain maps as depicted in \Cref{fig:generalKellerProofGlueingTwice}, which in turn disproves $p$ being a minimal factorization.
    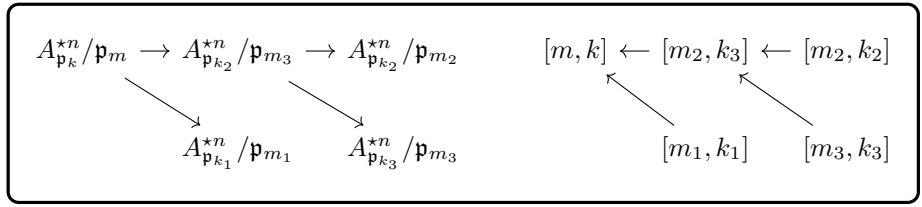
\begin{figure}[h]
        \centering
\begin{tikzcd}[style={every outer matrix/.append style={draw=black, inner xsep=6pt , inner ysep=9pt, rounded corners, very thick}},row sep=0.6cm,column sep=0.4cm]
A^{\star n}_{\fp_{k}}/\fp_{m} \arrow[r] \arrow[rd] & A^{\star n}_{\fp_{k_2}}/\fp_{m_3} \arrow[rd] \arrow[r] & A^{\star n}_{\fp_{k_2}}/\fp_{m_2} &   & {[m,k]} & {[m_2,k_3]} \arrow[l] & {[m_2,k_2]} \arrow[l]\\
                                                   & A^{\star n}_{\fp_{k_1}}/\fp_{m_1}                      & A^{\star n}_{\fp_{k_3}}/\fp_{m_3} &    &                              & {[m_1,k_1]}\arrow[lu]                      & {[m_3,k_3]}\arrow[lu]
\end{tikzcd}
   \caption{Canonical maps yielding further factorizations next to the corresponding intervals, whenever $m\le m_1\le k_1<m_2\le k_2<m_3\le k_3\le k$.}
        \label{fig:generalKellerProofGlueingTwice}
    \end{figure}
    
    Hence, $Q$ can not contain three or more factors. Now assume $Q$ contains a single factor, i.e., $[m_1,k_1]\in\cI(f)$. By \Cref{lem:contracting}, we have $m\le m_1\le k_1\le k$. Since these intervals are not equal, assume without loss of generality $m<m_1$. Then using \Cref{lem:contracting}, one can factor $\phi$ as $P\to k(\fp_m)\times Q\to Q$, which implies $p$ can not be a minimal factorization.

    \indent
    To summarize, we know that in the situation of \eqref{eqn:cDiagramGeneralizedKellerProof2} the following must hold:
    \[
    P=A^{\star n}_{\fp_{k}}/\fp_{m}, Q=A^{\star n}_{\fp_{k_1}}/\fp_{m_1}\times A^{\star n}_{\fp_{k_2}}/\fp_{m_2},\enspace m\le m_1\le k_1<m_2\le k_2\le k.
    \]
     By \Cref{fig:generalKellerProofGlueingTwice}, any number in between $k_1<m_2$ gives rise to a factorization of $p$, showing $p$ can not minimal. So $m_2$ must be the successor of $k_1$. Similar reasoning shows that we must have $m=m_1,k_2=k$, or in other words, the chains are of the form $[m,k]\in\cI(g)\setminus\cI(f)$ and $[m,k_1],[k_1+1,k]\in\cI(f)\setminus\cI(g)$.

     We are left to show that the gluing rule yields a minimal factorization. Consider the homological epimorphism $g\colon A^{\star n}\to S^\prime$, which corresponds to the chain $\{[m,k]\}\cup\cI(f)\setminus\{[m,k_1],[k_1+1,k]\}$. Then we certainly have a factorization by \Cref{lem:contracting}, which is the identity on intervals in $\cI(f)\cap\cI(g)$, and the factor corresponding to $[m,k]$ maps to factors of $[m,k_1],[k_1+1,k]$. Let $\tilde S$ correspond to the common factors of $S$ and $S^\prime$. To show that the factorization is minimal, assume $[g]\le [h]\le [f]$, i.e., the following diagram commutes:
    \begin{eqnarray}\label{eqn:cDiagramGeneralizedKellerProof3}
    \begin{tikzcd}
    S^\prime =\tilde S\times A^{\star n}_{\fp_k}/\fp_m\arrow[r, "p^\prime"'] \arrow[rr, "p", bend left=15] & T \arrow[r, "p^{\prime\prime}"']                     & \tilde S \times A^{\star n}_{\fp_{k_1}}/\fp_{m}\times A^{\star n}_{\fp_{k}}/\fp_{k_1+1}=S \\
                                                                   & A^{\star n}. \arrow[ru, "f"'] \arrow[lu, "g"] \arrow[u, "h"] &  
    \end{tikzcd}
    \end{eqnarray}
    We immediately notice that $T\neq0$. Similar reasoning as done in the case of \eqref{eqn:cDiagramGeneralizedKellerProof} shows that $T$ contains $S$ as a subring. In particular, we may restrict $p^\prime$ to a map $\phi\colon A^{\star n}_{\fp_k}/\fp_m \to P$, where $P$ is the product satisfying $T=\tilde S\times P$. \Cref{lem:contracting} tells us that this map is determined by partitions of $[m,k]$, because $p$ maps to $ A^{\star n}_{\fp_{k_1}}/\fp_{m}\times A^{\star n}_{\fp_{k}}/\fp_{k_1+1}$. If there are three or more blocks in these, they cannot map to $ A^{\star n}_{\fp_{k_1}}/\fp_{m}\times A^{\star n}_{\fp_{k}}/\fp_{k_1+1}$ by \Cref{lem:contracting}. Similarly, if there are two blocks or one, any partition than $\{[m,k]\}$ or $\{[m,k_1],[k_1+1,k]\}$ can not map to $S$ while mapping to all factors. But this precisely means $T=S^\prime$ or $T=S$, which proves that $p$ is a minimal factorization.
\end{proof}

Using this description of the minimal factorizations, we can now identify the meet-prime elements and as such the smashing spectrum of $A^{\star n}$.
\begin{theorem}\label{thm:smashing_spectrum_n}
Let $n \geq 1$.  Then $|\Spc^\mathrm{s}(\sfD(A^{\star n}))| = 2n+1$ and has a topology described as
    
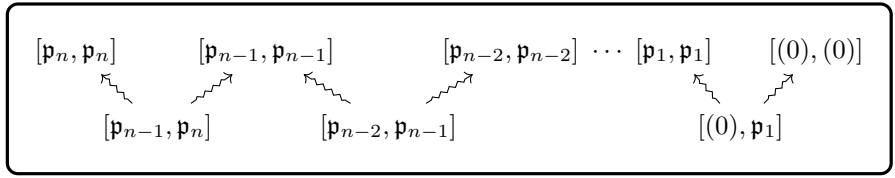
\[
        \begin{tikzcd}[style={every outer matrix/.append style={draw=black, inner xsep=6pt , inner ysep=9pt, rounded corners, very thick}},row sep=0.4cm,column sep=-0.5cm]
{[\fp_n,\fp_n]} && {[\fp_{n-1},\fp_{n-1}]} && {[\fp_{n-2},\fp_{n-2}]} &  \quad \cdots \quad & {[\fp_{1},\fp_{1}]} && {[(0),(0)]}\\
& {[\fp_{n-1},\fp_{n}]} \arrow[ru, rightsquigarrow] \arrow[lu, rightsquigarrow] && \arrow[ru, rightsquigarrow] \arrow[lu, rightsquigarrow]  [\fp_{n-2},\fp_{n-1}] &&&& \arrow[ru, rightsquigarrow] \arrow[lu, rightsquigarrow]  [(0),\fp_{1}] \end{tikzcd}
\]
where the top points are closed, the bottom points are open, and the arrows signify specialization closure.
\end{theorem}

\begin{proof}
To find the meet-prime elements of $\Sm(\sfD(A^{\star n}))$, it suffices to find those $\mathcal{I} \in \interc(A^{\star n})$ which are only minimally included in one element. Evidently, $\cI$ can not contain more than one element, as otherwise $\cI$ can be written as $\cI_1\lor \cI_2$, which corresponds to a meet in $\home(A^{\star n})$, by \Cref{prop:meet_and_join}. The minimal inclusions are characterized by the subset and gluing rule, see \Cref{prop:minimal_fact}. Let $\cI=\{[\fp_m,\fp_k]\}$, for $0\le m\le k\le n$. If $m=k$, the subset rule witnesses the minimal inclusions $\varnothing\subset\{[\fp_m,\fp_m]\}$, hence, $\{[\fp_m,\fp_m]\}$ is a join-prime. In other words, its corresponding class of homological epimorphisms is a meet-prime in $\home(A^{\star n})$ (due to the order-reversing bijection mentioned in \Cref{rem:frameordering}). Furthermore, the gluing rule tells us that the general singleton $\cI=\{[\fp_m,\fp_k]\}$ with $m\neq k$ arises only from tuples $\{[m,k_1],[k_1+1,k]\}$, where $m\le k_1<k$. As $\cI$ can only admit one such minimal inclusion, we must have $m=k_1,k_1+1=k$, i.e., $\cI=\{[m,m+1]\}$, for $0\le m<n$. Hence, $|\Spc^\mathrm{s}(\sfD(A^{\star n}))| = 2n+1$.

\indent
As for the specialization closure, we will show the more general statement $\cI\rightsquigarrow\cJ$ if and only if $[f_{(\cI)} ]\le[f_{(\cJ)}]$ in $\home(A^{\star n})$, where we use the notation of the bijection in \Cref{thm:bijectionsofframes}. In conjunction with \Cref{lem:contracting}, this proves the claimed specializations. In the following, let $\sfI,\sfJ$ denote their corresponding smashing ideal in $\Sm(\sfD(A^{\star n}))$ using the bijections of \Cref{thm:bijectionsofframes}. Then $\cI\rightsquigarrow\cJ$ if and only if
\begin{eqnarray}\label{eqn:smashing_spectrum}
    \sfJ\in D(a) \mbox{ for some }a\in\Sm(\sfD(A^{\star n}))\implies \sfI\in D(a),
\end{eqnarray}
where $D(a)$ is the basic open associated to the frame element $a$. Recall that by definition $D(a)=\{p\in \Spc^\mathrm{s}(\sfD(A^{\star n}))\mid a\not\le p\}$. We see that \eqref{eqn:smashing_spectrum} holds true, as otherwise $a\le \sfI\le\sfJ$ leads to a contradiction.
\end{proof}

The only remaining thing to discuss is the comparison map $\Spc^\mathrm{s}(\mathsf{D}(A^{\star n})) \to \Spc(\mathsf{D}(A^{\star n})^\omega)^\vee$ in terms of elements of $\interc(A^{\star n})$. Recall that the compactly generated ideals are those of the form $[(0), \fp_i]$. Then by unravelling the relevant definitions, we arrive at the following:

\begin{lemma}
Let $n \geq 1$. Then the comparison map $\Spc^\mathrm{s}(\mathsf{D}(A^{\star n})) \to \Spc(\mathsf{D}(A^{\star n})^\omega)^\vee$ is defined on points at the level of chains as:
\begin{itemize}
	\item $[\fp_i, \fp_i] \mapsto [(0), \fp_i]$.
	\item $[\fp_{i-1}, \fp_i] \mapsto [(0), \fp_i]$.
\end{itemize}
\end{lemma}

\section{Mixing idempotent and non-idempotent ideals}\label{sect:mixing}

In the previous section, we constructed a family of rings $A^{\star n}$ and discussed the tensor-triangular properties of $\sfD(A^{\star n})$. In this section, we will construct a “model” for any valuation domain with finite Zariski spectrum. By $A^{\star n}$ we denoted a valuation domain with Krull dimension $n$, such that all of its prime ideals are idempotent. One also has the other extreme, where none of the non-zero prime ideals are idempotent. The goal of this section is to study those rings with a finite, totally ordered spectrum, such that the idempotence of their prime ideals is prescribed. Determining the smashing frame of those rings, yields all possible smashing frames which can arise from valuation domains with finite Krull dimension.

\begin{proposition}\label{prop:mixedKeller}
Let $n \geq 1$, $\idem\colon \N_{\le n}\to\F_2$ be any map with $\idem(0)=1$, and $k$ any field. Then there exists a valuation domain $A_{\idem}$ (with value group $G_\idem$), such that its chain of all prime ideals 
\[
(0)=\fp_0\subset\fp_1\subset\hdots\subset \fp_n
\]
satisfies the following: $\fp_j$ is idempotent if and only if $\idem(j)=1$, non-idempotent if and only if $\idem(j)=0$, and $A_{\idem}/\fp_n\cong k$. 
\end{proposition}
\begin{proof}
    Fix some $\ell\ge2$, and define for $1\le i\le n$ the totally ordered abelian group
    \[
    G_i\coloneqq\begin{cases}\Z, &\idem(i)=0,\\\Z\left[\frac{1}{\ell}\right], &\idem(i)=1.
    \end{cases}
    \]
    Then the value group of $A_\idem$ will be $G=G_\idem=\bigoplus\limits_{i=1}^{n}G_i$, endowed with the lexicographical product as described in \Cref{def:lexicographicSupport}. To prove the statement, we will classify the prime filters $F_j$ in $G^+$, which gives us the desired ring $A_{\idem}=R_k(G_\idem)$ by \Cref{prop:idealsAndFilters} and \Cref{prop:prescribed}. This is proven exactly as in the proof of \Cref{prop:structureofprimes}. Analogously as for the idempotence of these prime filters, we reduce to the following statement: 
\[
\mbox{Every }(v_1,\hdots,v_j)\in\left(\bigoplus\limits_{k=1}^jG_i\right)_{>0}\mbox{ can be written as the sum of two elements in }\left(\bigoplus\limits_{k=1}^jG_i\right)_{>0}.
\]
Assume $\supp(v)=k<j$, i.e., $v=(0,\hdots,0,v_k,v_{k+1},\hdots,v_j)$, and $v_k>0$. If there exists some $k<k^\prime\le j$ such that $v_{k^\prime}>0$, write
\[
v=(0,\hdots,0,v_k,v_{k+1},\hdots,v_{k^\prime-1},0,\hdots,0)+(0,\hdots,0,v_{k^\prime},\hdots,v_j).
\]
If there exists some $k<k^\prime\le j$ such that $v_{k^\prime}<0$, write
\[
v=(0,\hdots,0,v_k,v_{k+1},\hdots,v_{k^\prime-1},2v_{k^\prime},0,\hdots,0)+(0,\hdots,0,-v_{k^\prime},\hdots,v_j).
\]
If $v_k$ is the only non-zero entry, we may write $v=(0,\hdots,0,v_k,-1,0\hdots,0)+(0,\hdots,0,1,0\hdots,0)$. Finally, the case $\supp(v)=j$ follows from the fact that $G_j=\Z\left[\frac{1}{\ell}\right]$, and our calculation in \Cref{ex:keller=2}, see \eqref{eq:thickening}. 

We are left to show why the prime filters $F_j$ are not idempotent for all $j$ with $\idem(j)=0$. To this end, we will exhibit an element in $F_j$ which cannot be written as the sum of two elements in $F_j$. For $j\le n$ consider the element $m^j=(m_1^j,\hdots)\in G$, which is defined as
\[
m^j_i=\begin{cases}1,&\mbox{if}\,\, i=j\\0,&\mbox{else}.\end{cases}
\]
This element cannot be written as the sum of two elements in $F_j$ for all $j$ with $\idem(j)=0$ due to the discreteness of the summand $\Z$ at index $j$, and since any $v,w\in F_j\subset G^+$ satisfy $\supp(v+w)=\min(\supp(v),\supp(w))$, see \Cref{def:lexicographicSupport}. Thus, $F_j$ is non-idempotent.
\end{proof}
\begin{remark}
    The analogous statement of \Cref{prop:mixedKeller} holds for $n=\infty, \idem\colon\N_{\geq 1}\cup\{\infty\}\to\F_2$, where $\idem(0)=\idem(\infty)=1$. Indeed, the maximal ideal $\fp_\infty$ will always be idempotent, as the union of a strictly ascending chain of prime ideals is always an idempotent prime ideal.
\end{remark}

We see that \Cref{ex:keller} (resp., \Cref{ex:keller=2}) corresponds to the case $n=1$ (resp., $n=2$) and the map $\idem\colon\N_{\le n}\to \F_2$ being constantly $1$. We will discuss a non-constant example thoroughly in \Cref{ex:mixed}.

A key ingredient for determining the smashing spectrum of $\sfD(A^{\star n})$ were the minimal factorizations between homological epimorphisms with domain $A^{\star n}$. Hence, to generalise the former, we will also generalise the latter. The case $\idem\equiv 1$ in the following result recovers the subset and gluing rule as proven in \Cref{prop:minimal_fact}.
\begin{proposition}\label{prop:minimal_fact_mixed}
Let $n\in\N_{\ge1}$ and let $\idem\colon \N_{\le n}\to\F_2$ be any map with $\idem(0)=1$, $k$ any field. Let $f\colon A_\idem\to S$, $g\colon A_\idem\to S^\prime$ be homological epimorphisms. If $[f]\neq[g]$, then there is a minimal factorization $[g]\le[f]$ if and only if there is solely one interval $[m,k]\in\cI(g)\setminus\cI(f)$ and a number $m\le k_1\le k$, such that the chain $I$ given by
        \[
        I=\begin{cases}
            \varnothing, &\mbox{ if }m=k,\\
            \{[m,k-1]\}, &\mbox{ if }k-1=k_1, [m,k-1]\in\cI(f), [k,k]\not\in\cI(f),\\
            \{[m,k_1],[k_1+1,k]\}, &\mbox{ if }m\le k_1<k, \mbox{ and } I\subseteq\cI(f),
        \end{cases}
        \]
        exists, and satisfies $\cI(g)=\{[m,k]\}\cup(\cI(f)\setminus I)$.
\end{proposition}
\begin{proof}
    Firstly, notice that \Cref{lem:contracting} always yields a factorization $[g]\le[f]$, whenever the chain $I$ exists with the claimed properties of the statement.
    Now, without loss of generality, we may assume that $S$ and $S^\prime$ are products as constructed in \Cref{constr:interc}. Assume there is a factorization $[g]\le[f]$ witnessed by $p\colon S^\prime\to S$ and let the ring $\tilde S$ correspond $\cI(f)\cap\cI(g)$, $P$ to $\cI(g)\setminus \cI(f)$, and $Q$ to $\cI(f)\setminus \cI(g)$. Then the following diagram commutes:
       \begin{eqnarray}\label{eqn:cDiagramMixedKeller}
    \begin{tikzcd}
        S^\prime=\tilde S\times P \arrow[rr, "p"] &                                              &  S=\tilde     S\times Q \\
                                          & A_{\idem}. \arrow[lu, "g"] \arrow[ru, "f"'] &                   
    \end{tikzcd}
    \end{eqnarray}
    Moreover, $P$ or $Q$ is non-zero, and $P\neq Q$. We have two cases, in each of which we will show that $p$ is minimal if and only if the chain $I$ exists as claimed.
    \begin{enumerate}
        \item $Q=0$. Then $\cI(f)\subsetneq\cI(g)$, and $P\neq0$. We now proceed just as in the proof of the subset rule (\Cref{prop:minimal_fact}). So $p$ is minimal if and only if $m=k$, which covers the first case of $I=\varnothing$.
        \item $Q\neq0$. Since $Q\neq P$, we must have $\cI(f)\not\subseteq\cI(g)$. Similarly, we can not have $\cI(f)\supseteq\cI(g)$, as such inclusion would contradict to our findings in the previous case. In particular, $P\neq0$. As in the proof of the gluing rule in \Cref{prop:minimal_fact}, $P$ has at most one factor, and $Q$ at most two. Moreover, we may assume $p$ can be written as $(\mathrm{id}_{\tilde S},\phi)$, where $\phi\colon P\to Q$ is the canonical map. Since $P,Q\neq0$, \Cref{lem:contracting} allows us to write 
         \[
    P=(A_{\idem})_{\fp_{k}}/\fp_{m}, Q=(A_{\idem})_{\fp_{k_1}}/\fp_{m}\times \tilde Q,\enspace m\le k_1\le k\le n, m\neq k.
    \]
    Here $\tilde Q$ might be $0$, or a single non-zero factor. We claim that the factorization witnessed by $p$ is minimal if and only if
    \[
    \tilde Q=\begin{cases}
        0, &\mbox{ if }k-1=k_1,[k,k]\not\in\cI(f),\\
        (A_{\idem})_{\fp_k}/\fp_{k_1+1}, &\mbox{ if }m\le k_1<k, [k_1+1,k]\in\cI(f). 
        
    \end{cases}
    \]
    The case $k-1=k_1,[k,k]\not\in\cI(f)$ is immediate, as there is only one factorization $p$ by \Cref{lem:contracting}, and $\tilde Q=0$. Hence, the case where $I=\{[m,k-1]\}$. On the other hand, if $m\le k_1<k, [m,k_1],[k_1+1,k]\in\cI(f)$, the statement follows from a similar argument as in the proof of the gluing rule in \Cref{prop:minimal_fact}. In particular, $I=\{[m,k_1],[k_1+1,k]\}$.\qedhere\end{enumerate}
\end{proof}
With the description of minimal inclusions in $\home(A_\idem)$ by \Cref{prop:minimal_fact_mixed}, we will now generalise \Cref{thm:smashing_spectrum_n} and classify the smashing spectrum of $\sfD(A_\idem)$. The following definitions will provide us with the language that we will require to do this.
\begin{definition}\label{def:next}
    Let $R$ be a valuation domain of finite Krull dimension $n$ with chain of prime ideals $(0)\subsetneq\fp_1\subsetneq\hdots \subsetneq\fp_n$. Define the map $\Next \colon \N_{\le n} \to \N_{\le n}$ as
    \[
    \Next(i)\coloneqq\begin{cases}
    k, &\mbox{ if }\fp_i\subsetneq\fp_k, \idem(k)=1,\mbox{ and }\idem(j)=0\mbox{ for all }i<j<k,\\
    n, &\mbox{ otherwise}.
    \end{cases}     
    \]
    We denote the $k$-fold iteration of this self-map as $\Next^k$ and let $\Next^k_0\coloneqq\Next^k(0)\in\N$ with the convention that $\Next^0_0=0$. Finally we define
    \begin{align*}
        M\coloneqq&\max\{i\mid \idem(i)=1\}\ge0, \\
        d\coloneqq&\max\{n-M,\vert \Next_0^{k+1}-\Next_0^k\vert\colon 0\le k<n\}.
    \end{align*}
\end{definition}

Informally speaking, $\Next(i)$ returns the minimal index above $i$, such that its corresponding prime ideal is idempotent, unless there is none, in which case the map returns $n$. The value $M$ is the maximal index among all idempotent prime ideals, and $d$ is the maximal amount of subsequent non-idempotent prime ideals.

\begin{theorem}\label{thm:smashing_spectrum_mixed}
    Let $R$ be a valuation domain with finite Krull dimension $n$. In notation of \Cref{def:next}, let $[m,k]$ denote the interval $[\fp_m,\fp_k]$. Then $\Spc^\mathrm{s}(\sfD(R))$ can be described as 
    \begin{eqnarray*}
\begin{tikzcd}[style={every outer matrix/.append style={draw=black, inner xsep=6pt , inner ysep=9pt, rounded corners, very thick}},row sep=0.43cm,column sep=0.3cm]
{[M,M]}& {[\Next_0^2,\Next_0^2]} &&& {[\Next_0^1,\Next_0^1]} &&  {[0,0]} \\
\hdots\arrow[u, rightsquigarrow]  \arrow[ru, rightsquigarrow] & {[\Next_0^1,\Next_0^2]}\arrow[u, rightsquigarrow]\arrow[r, rightsquigarrow]&  \hdots\arrow[r, rightsquigarrow]  & {[\Next_0^1,\Next_0^1+1]}\arrow[ru, rightsquigarrow]& {[0,\Next_0^1]}\arrow[u, rightsquigarrow]\arrow[r, rightsquigarrow]
& \hdots\arrow[r, rightsquigarrow] & {[0,1]}\arrow[u, rightsquigarrow] \\
{[M,M+1]}\arrow[uu, rightsquigarrow, bend left=35]&\hdots\arrow[l, rightsquigarrow]&{[M,n]}\arrow[l, rightsquigarrow]&&&&
\stepcounter{equation}\addtocounter{equation}{-1}{(\theequation)}\addtocounter{equation}{-1}\refstepcounter{equation}\label{eqn:smashing_spectrum_mixed}
\end{tikzcd}
    \end{eqnarray*} 
where the arrows signify
specialization closure.
In particular, the dimension of $\Spc^\mathrm{s}(\sfD(R))$ is $d+1$, and 
\[
\vert \Spc^\mathrm{s}(\sfD(R))\vert = \sum_0^n\idem(i)+\sum_{\idem(i)=1}\left(\Next_0^{i+1}-\Next_0^i\right)+n-M.
\]
\end{theorem}
\begin{proof}
    Without loss of generality, we replace $R$ by $A_\idem$, for a suitable map $\idem\colon \N_{\le n}\to\F_2$. The proof of the specialization closures is identical to its proof in \Cref{thm:smashing_spectrum_n}. Hence, $\cI\rightsquigarrow\cJ$ if and only if $[f_{(\cI)} ]\le[f_{(\cJ)}]$ in $\home(A_\idem)$. Likewise, the meet-prime elements in $\Sm(\sfD(A_\idem))$ correspond to singletons $\cI\in\interc(A_\idem)$, which are minimally included in one element. For $0\le m\le k\le n$ and $\idem(m)=1$, let $\cI=\{[m,k]\}$. Then \Cref{prop:minimal_fact_mixed} classifies all minimal inclusions. 
    
    \indent
    If  $m=k$, the subset rule, i.e., the first case in \Cref{prop:minimal_fact_mixed}, witnesses $\cI$ corresponding to a meet-prime in $\Sm(\sfD(A_\idem))$. This explains the points of the form $[m,m]$, for which $\idem(m)=1$. 
    Now let $m\le k_1\le k$, $m\neq k$. If $k_1=k-1$, there is a minimal inclusion $\cI\ge\{[m,k-1]\}$ in $\interc(A_\idem)\cup\{\varnothing\}$ if and only if $\idem(k)=0$. Thus, by the order-reversing of the bijection in \Cref{thm:bijectionsofframes}, we see that $\{[m,k]\}$ corresponds to a meet-prime, whenever $m<k,\idem(k)=0$. Finally, in the case $m\le k_1<k$, there is a minimal inclusion $\cI\ge\{[m,k_1],[k_1+1,k]\}$. Hence, $\cI$ is a corresponds to a meet-prime if and only if there is one such possible choice for $m\le k_1<k$. In general, there are $k-m$ choices, but each non-idempotent prime in between $m+1\le k$ reduces the amount of choices by one. Therefore, the singletons of the form $\{[\Next_0^\ell,\Next_0^{\ell+1}]\}$ correspond to meet-primes in $\Sm(\sfD(A_\idem)$. This concludes the proof of the meet-primes of $\Sm(\sfD(A_\idem))$, and the arrows follow from \Cref{lem:contracting} using $\cI\rightsquigarrow\cJ$ if and only if $[f_{(\cI)} ]\le[f_{(\cJ)}]$ in $\home(A_\idem)$.

    \indent
     The claimed dimension and formula for the number of points in $\Spc^\mathrm{s}(\sfD(R))$ then follow from a simple count using the chains.
\end{proof}
\addtocounter{thm}{1}
\begin{remark}
    Similar to \cref{cor:counting}, it is possible to write down a formula for $|\Sm(\sfD(A_{\idem}))|$. Once again, this can be achieved by considering those elements of $\interc(A_{\idem})$ of a given size. For example, the number of elements of size two is given by
    \[
    \sum_{i_1 \mid \idem(i_1)=1} \, \sum_{\substack{i_2 \mid \idem(i_2)=1\\i_1 < i_2}} (i_2-i_1)(n+1-i_2).
    \]
    However, we have been unable to obtain a simple closed formula for $|\Sm(\sfD(A_{\idem}))|$.
\end{remark}

\begin{example}\label{ex:mixed}
    Consider the map $\idem\colon\N_{\le 5}\to \F_2$, which only maps $0$ and $4$ to $1$. By \Cref{prop:mixedKeller} we obtain a valuation domain $A=A_{\idem}$ with 6 prime ideals, of which only $(0)$ and $\fp_4$ are idempotent in the chain of all prime ideals of $R$: $(0)\subset \fp_1\subset \fp_2\subset \fp_3\subset \fp_4\subset\fm$.
    By \Cref{thm:smashing_spectrum_mixed} the smashing spectrum is given as follows:
        \[
\begin{tikzcd}[style={every outer matrix/.append style={draw=black, inner xsep=6pt , inner ysep=9pt, rounded corners, very thick}},row sep=0.4cm,column sep=0.3cm]
 & \{[\fp_4,\fp_4]\} &&&&&  \{[(0),(0)]\} \\
 \{[\fp_4,\fm]\} \arrow[ru, rightsquigarrow] &&  \arrow[lu, rightsquigarrow]  \{[(0),\fp_4]\} \arrow[r, rightsquigarrow]  & \{[(0),\fp_3]\}\arrow[r, rightsquigarrow]& \{[(0),\fp_2]\}\arrow[r, rightsquigarrow]
& \{[(0),\fp_1]\}\arrow[ru, rightsquigarrow] & \end{tikzcd}
\]

Using \Cref{thm:bijectionsofframes} and applying \Cref{constr:interc} to $A$, we can identify all elements in $\interc(A)$ and $\home(A)$ in \Cref{tab:mixedExample}. Adopting the notation used in \Cref{ex:keller=2} for \Cref{fig:n=2}, we end up with the smashing frame $\Sm(\sfD(A))$ as depicted in \Cref{fig:mixedExample}.
\begin{figure}[H]\TopFloatBoxes
\begin{floatrow}
\centering
\ffigbox{
\begin{tikzcd}[style={every outer matrix/.append style={draw=black, inner xsep=6pt , inner ysep=9pt, rounded corners, very thick}},row sep=0.38cm,column sep=0.15cm,ampersand replacement=\&]
\Circled{0}        \&                                                                 \&                                            \\
\Circled[inner color=purple]{Q} \arrow[u]        \&  {\color{purple}k(\fp_4)} \arrow[lu]            \&                                            \\
\Circled[inner color=purple]{A_{\fp_1}} \arrow[u]\& Q\times k(\fp_4) \arrow[lu]\arrow[u]         \& {\color{purple}A/\fp_4} \arrow[lu]           \\
\Circled[inner color=purple]{A_{\fp_2}} \arrow[u]\& A_{\fp_1}\times k(\fp_4) \arrow[lu]\arrow[u]  \& Q\times A/\fp_4 \arrow[lu]\arrow[u]         \\
\Circled[inner color=purple]{A_{\fp_3}} \arrow[u]                    \& A_{\fp_2}\times k(\fp_4) \arrow[lu]\arrow[u]  \& A_{\fp_1}\times A/\fp_4 \arrow[lu]\arrow[u]  \\
                                                 \& A_{\fp_3}\times k(\fp_4)  \arrow[lu]\arrow[u]  \& A_{\fp_2}\times A/\fp_4 \arrow[lu]\arrow[u]  \\
                                                 \& \Circled[inner color=purple]{A_{\fp_4}}\arrow[u]                                   \& A_{\fp_3}\times A/\fp_4 \arrow[lu]\arrow[u] \\
                                                 \&                                                                 \& \Circled{A}  \arrow[lu]\arrow[u]                                        
\end{tikzcd}}
{%
  \caption{$\Sm(\sfD(A))$ with the notation of \Cref{fig:n=2}.}\label{fig:mixedExample}
}
\capbtabbox{%
\begin{tabular}{||P{.45\textwidth}||}
\hline
\begin{tabular}{P{.19\textwidth}|P{.26\textwidth}}
$\cI$ &Repr.\ in $\home(A)$
\end{tabular}
\\
\hline
{\renewcommand{\arraystretch}{1.3}\begin{tabular}{P{.19\textwidth}|P{.26\textwidth}}
$\varnothing$ & $A\to 0$\\
$\{[(0),(0)]\}$ & $A\to Q$\\
$\{[(0),\fp_1]\}$ & $A\to A_{\fp_1}$\\
$\{[(0),\fp_2]\}$ & $A\to A_{\fp_2}$\\
$\{[(0),\fp_3]\}$ & $A\to A_{\fp_3}$\\
$\{[(0),\fp_4]\}$ & $A\to A_{\fp_4}$\\
$\{[(0),\fm]\}$ & $A\to A$\\
$\{[\fp_4,\fp_4]\}$ & $A\to k(\fp_4)$\\
$\{[\fp_4,\fm]\}$ & $A\to A/\fp_4$\\
\end{tabular}}
\\
\hline
{\renewcommand{\arraystretch}{1.3}\begin{tabular}{P{.19\textwidth}|P{.26\textwidth}}
 $\{[(0),(0)],[\fp_4,\fp_4]\}$ &  $A\to Q\times k(\fp_4)$\\
 $\{[(0),(0)],[\fp_4,\fm]\}$ &  $A\to Q\times A/\fp_4$\\
 $\{[(0),\fp_1],[\fp_4,\fp_4]\}$ &  $A\to A_{\fp_1}\times k(\fp_4)$\\
 $\{[(0),\fp_1],[\fp_4,\fm]\}$ &  $A\to A_{\fp_1}\times A/\fp_4$\\
  $\{[(0),\fp_2],[\fp_4,\fp_4]\}$ &  $A\to A_{\fp_2}\times k(\fp_4)$\\
 $\{[(0),\fp_2],[\fp_4,\fm]\}$ &  $A\to A_{\fp_2}\times A/\fp_4$\\
  $\{[(0),\fp_3],[\fp_4,\fp_4]\}$ &  $A\to A_{\fp_3}\times k(\fp_4)$\\
 $\{[(0),\fp_3],[\fp_4,\fm]\}$ &  $A\to A_{\fp_3}\times A/\fp_4$
\end{tabular}}
\\
\hline  
\end{tabular}
}{
  \caption{The empty chain along with the 16 elements of $\interc(A)$, and their representative in $\home(A)$.}\label{tab:mixedExample}
}
\end{floatrow}
\end{figure}

\end{example}
Since $A_\idem$ generalises the family of examples given by the rings $A^{\star n}$, one might ask about further relations between them and their smashing frames. We will conclude this paper with a comparison between the frames $\Sm(\sfD(A_\idem))$ and $\Sm(\sfD(A^{\star n}))$. 
\begin{theorem}\label{thm:mixed_subframe}
    Let $n\in\N_{\ge1}$, $\idem\colon \N_{\le n}\to\F_2$ be any map with $\idem(0)=1$, and $k$ any field. Let $G$ denote the value group of $A^{\star n}$, i.e., the $n$-fold lexicographic product of $\Z\left[\frac{1}{\ell}\right]$. Then the group $G_\idem$ constructed in \Cref{prop:mixedKeller} yields a canonical inclusion of $G_\idem\hookrightarrow G$, which induces a ring extension $\varphi_\idem\colon A_\idem\hookrightarrow A^{\star n}$ using  \Cref{lemma:functorialityVDconstruction}, allowing us to identify $A_\idem\subseteq A^{\star n}$.
    Moreover, the induced map on the Zariski spectra is bijective and order-preserving. 

Under this identification, $\Sm(\sfD(A_\idem))$ (resp., $\home(A_\idem)$) is a subframe of $\Sm(\sfD(A^{\star n}))$ (resp., $\home(A^{\star n})$).
\end{theorem}
\begin{proof}
    The constructed morphism $\varphi_\idem$ of \Cref{lemma:functorialityVDconstruction} is evidently injective, as it is the restriction of a morphism between fields. Moreover, using the explicit characterization of prime filters in $G_\idem,G$ and naturality mentioned in \Cref{lemma:functorialityVDconstruction}, it is straightforward to deduce that $\Spec(\varphi_\idem)$ sends any prime ideal $\fp_k\subset A^{\star n}$ with $\mathrm{ht}(\fp_k)=k$ to the prime ideal with height $k$ in $A_\idem$. Hence, $\Spec(\varphi_\idem)$ is an order-preserving bijection. Thus, if we denote intervals $[\fp_m,\fp_k]$ by $[m,k]$, the inverse of $\Spec(\varphi_\idem)$ induces the injection
    \[
\begin{tikzcd}[row sep=0.1cm,column sep=0.5cm]
\inter(A_\idem) \arrow[rr]  &  & \inter(A^{\star n}) \\
{[m,k]} \arrow[rr, maps to] &  & {[m,k]},            
\end{tikzcd}
    \]
which in turn induces an injection $\iota\colon\interc(A_\idem)\cup\{\varnothing\} \hookrightarrow \interc(A^{\star n})\cup\{\varnothing\}$. Moreover, one sees that the constructions $\cI\wedge \cJ,\cI\lor\cJ$ of \Cref{lem:constructing_new_chains} yield operations in $\interc(A_\idem)\cup\{\varnothing\}$. 
Therefore, the injection $\iota$ preserves meets and joins. Hence, $\iota$ also preserves the order, since $\cI\le\cJ$ if and only if $\cI\wedge\cJ=\cI$. Conclusively, $\iota$ is a monomorphism of frames, stemming from the identification of $\varphi_\idem$. As such, the order-reversing and order-preserving bijections of \Cref{thm:bijectionsofframes} yield the final claims.
\end{proof}

\bibliography{bibliography}\bibliographystyle{alpha}
\end{document}